\documentclass{article}
\usepackage[utf8]{inputenc}
\usepackage{savesym}
\usepackage{amsthm,amssymb,amsfonts,amsmath,bm,newtxmath}
\usepackage[usenames,dvipsnames]{xcolor}

\usepackage{faktor}
\savesymbol{bigtimes}
\usepackage{mathabx}
\usepackage{enumitem}
\restoresymbol{NEWbigtimes}{bigtimes}
\usepackage[margin=1in]{geometry}
\usepackage[textsize=tiny]{todonotes}
\usepackage{tikz-cd}
\usetikzlibrary{cd}
\usepackage{epstopdf}

\newtheorem*{theorem*}{Theorem}
\newtheorem{theorem}{Theorem}
\newtheorem{lemma}[theorem]{Lemma}
\newtheorem{proposition}[theorem]{Proposition}
\newtheorem{corollary}[theorem]{Corollary}
\theoremstyle{definition}
\newtheorem{definition}[theorem]{Definition}
\theoremstyle{remark}
\newtheorem{remark}[theorem]{Remark}

\newtheorem{question}[theorem]{Question}
\newtheorem{example}[theorem]{Example}
\newtheorem{conjecture}[theorem]{Conjecture}

\numberwithin{theorem}{section}

\newcommand\cP{{\mathcal P}}

\newcommand\CC{{\vmathbb C}}
\newcommand\FF{{\vmathbb F}}
\newcommand\HH{{\vmathbb H}}
\newcommand\KK{{\vmathbb K}}
\newcommand\NN{{\vmathbb N}}
\newcommand\PP{{\vmathbb P}}
\newcommand\QQ{{\vmathbb Q}}
\newcommand\RR{{\vmathbb R}}
\renewcommand\SS{{\vmathbb S}}
\newcommand\TT{{\vmathbb T}}
\newcommand\WW{{\vmathbb W}}
\newcommand\ZZ{{\vmathbb Z}}

\newcommand\ba{{\bm a}}
\newcommand\bb{{\bm b}}
\newcommand\bc{{\bm c}}
\newcommand\bd{{\bm d}}
\newcommand\be{{\bm e}}

\newcommand\bu{{\bm u}}
\newcommand\bv{{\bm v}}
\newcommand\bw{{\bm w}}

\newcommand\by{{\bm y}}
\newcommand\bz{{\bm z}}
\newcommand\bX{{\bm X}}

\newcommand\blambda{{\bm \lambda}}

\newcommand{\0}{{\vmathbb 0}}
\newcommand{\1}{{\vmathbb 1}}
\newcommand{\zerovec}{{\underline{\vmathbb 0}}}

\newcommand\SetOf[2]{\left\{\left.#1\vphantom{#2}\ \right|\ #2\vphantom{#1}\right\}}
\newcommand{\overbar}[1]{\mkern 1.5mu\overline{\mkern-2.5mu#1\mkern-1.5mu}\mkern 2.5mu}

\newcommand\pseries[1]{#1\{\!\{t\}\!\}}
\newcommand\hseries[2]{#1\llbracket t^{#2} \rrbracket}
\newcommand\fp[2]{\big((#1) \; , \, (#2)\big)} 

\DeclareMathOperator{\lc}{lc}

\DeclareMathOperator{\mult}{mult}
\DeclareMathOperator{\sgn}{sgn}
\DeclareMathOperator{\ph}{ph}
\DeclareMathOperator{\val}{val}
\DeclareMathOperator{\sval}{sval}
\DeclareMathOperator{\fval}{fval}
\DeclareMathOperator{\phval}{phval}
\DeclareMathOperator{\trop}{trop}
\DeclareMathOperator{\ftrop}{ftrop}
\DeclareMathOperator{\aff}{aff}

\DeclareMathOperator{\initial}{in}

\title{Geometry of tropical extensions of hyperfields}
\author{James Maxwell\thanks{School of Mathematics, University of Bristol. E-mail: \texttt{james.maxwell@bristol.ac.uk}} \and Ben Smith\thanks{School of Mathematical Sciences, Lancaster University. E-mail: \texttt{b.smith9@lancaster.ac.uk}}}

\begin{document}
\maketitle

\begin{abstract}
We study the geometry of tropical extensions of hyperfields, including the ordinary, signed and complex tropical hyperfields.
We introduce the framework of `enriched valuations' as hyperfield homomorphisms to tropical extensions, and show that a notable family of them are relatively algebraically closed.
Our main results are hyperfield analogues of Kapranov's theorem and the Fundamental theorem of tropical geometry.
Utilising these theorems, we introduce fine tropical varieties and prove a structure theorem for them in terms of their initial ideals.
\end{abstract}

\section{Introduction}

Tropical geometry is a branch of combinatorial algebraic geometry that offers two approaches to studying algebraic varieties over valued fields $(\FF, \val)$.
Given an ideal $I \subseteq \FF[X_1, \dots, X_n]$, one approach is to study the variety $V(I)$ by considering its coordinate-wise image $\val(V(I)) \subseteq \RR^n$ in the valuation map.
An alternative approach is to consider $\val_*(I)$, the collection of polynomials over the tropical semiring $(\RR, \max, +)$ obtained by valuating the coefficients of all polynomials in $I$.
Given this tropicalised ideal, one can obtain its corresponding tropical variety $V(\val_*(I))$ in a more combinatorial manner.
The Fundamental theorem of tropical geometry~\cite{MS:21} states that $\val(V(I))$ and $V(\val_*(I))$ are equal, and so these two approaches are equivalent.

Hyperfields are structures that generalise fields by allowing the addition operation to be multivalued, i.e. the sum $a\boxplus b$ may be a set rather than a singleton.
Introduced by Krasner in the 50s~\cite{Kra:56, Kra:83}, they rose to prominence within tropical geometry via the articles of Viro~\cite{Viro:10,Viro:11} and the introduction of matroids over hyperfields~\cite{Baker+Bowler:18}.
Viro noted that enriching the tropical semiring with a hyperfield structure allowed tropical varieties to be defined as genuine algebraic varieties over the tropical hyperfield $\TT$.
This approach via hyperfields gives a number of other advantages.
For example, hyperfields offer a very flexible framework for constructing `tropical-like' spaces known as \emph{tropical extensions}.
Moreover, valuations can be rephrased as hyperfield homomorphisms to the tropical hyperfield.
This perspective allows one to define a wider family of `valuation-like' maps as homomorphisms to tropical extensions, that we refer to as \emph{enriched valuations}.
Our goal in this article is to effectively layout the framework of tropical extensions and enriched valuations, and to extend a number of key theorems of tropical geometry to this setting.

\subsection{Our results}
Given a hyperfield $\HH$ and an ordered abelian group $\Gamma$, one can define a new hyperfield $\HH \rtimes \Gamma$ called the \emph{tropical extension} of $\HH$ by $\Gamma$. 
This can be viewed as a tropical hyperfield with `coefficients' in $\HH$: see Definition~\ref{def:tropical+extension} for a precise definition.
Key examples of this construction include the tropical hyperfield $\TT \cong \KK \rtimes \RR$, the tropical extension of the Krasner hyperfield $\KK$ by $(\RR,+)$, and the signed tropical hyperfield $\SS \rtimes \RR$, the tropical extension of the hyperfield of signs $\SS$ by $(\RR,+)$.
Tropical extensions of semirings were pioneered in the articles of Akian, Gaubert and Gutermann~\cite{AGG:09, AGG:14} but have proved useful tools for studying hyperfields, such as proving classification results~\cite{BowlerSu:21}.
Moreover, they provide a framework for defining enriched valuations, which we lay out in Section~\ref{subsec:homo+val}.
These enriched valuations are hyperfield homomorphisms from a field $\FF$ to a tropical extension $\HH \rtimes \Gamma$ where $\HH$ records additional information about an element beyond its order.
The prototypical example of an enriched valuation is the signed valuation $\sval\colon\FF \rightarrow \SS \rtimes \Gamma$ that also records the sign of an element. 

A key property of a valuation map $\val \colon \FF \rightarrow \TT$ is that it is \emph{relatively algebraically closed}: for each polynomial $p \in \FF[X]$ such that the induced tropical polynomial $\val_*(p) \in \TT[X]$ has a tropical root $b \in \TT$, there exists a lift $a \in \val^{-1}(b)$ that is a root of $p$.
This property gives rise to a proof of Kapranov's theorem: for any polynomial $p \in \FF[X_1, \dots, X_n]$, the tropical hypersurface $V(\val_*(p))$ and the image $\val(V(p))$ of the algebraic hypersurface $V(p)$ coincide.
It has also been well studied in commutative algebra and $p$-adic analysis as Hensel's Lemma.
It was formulated for hyperfields in~\cite{Maxwell:21} as a tool for generalising Kapranov's theorem to all relatively algebraically closed maps $f \colon \HH \rightarrow \TT$.
Our first main theorem is showing that one can extend Kapranov's theorem to relatively algebraically closed maps between arbitrary hyperfields, not just those onto $\TT$.

To effectively state this, we briefly introduce some necessary concepts of algebraic geometry over hyperfields: see Section~\ref{sec:hyperfields} for full details.
Given a polynomial $p \in \HH[X_1, \dots, X_n]$ over a hyperfield, its evaluation at a point is a subset of $\HH$ as sums are multivalued.
We define its associated \emph{hypersurface} $V(p)$ to be the set of points $\ba \in \HH^n$ such that $\0 \in p(\ba)$.
If $p \in \HH[X]$ is a univariate polynomial, then we refer to $V(p)$ as the \emph{roots} of $p$.
Given a homomorphism of hyperfields $f \colon \HH_1 \rightarrow \HH_2$, a polynomial $p \in \HH_1[X_1, \dots, X_n]$ induces a polynomial $f_*(p) \in \HH_2[X_1, \dots, X_n]$ called the \emph{push-forward} of $p$.
We say a homomorphism is \emph{relatively algebraically closed} (RAC) if for all univariate polynomials $p \in \HH[X]$, the roots of its push-forward $f_*(p)$ pull-back to roots of $p$ (Definition~\ref{def:RAC}).

\begin{theorem*}[\ref{thm:Kapranov}]
Let $\HH_1, \HH_2$ be arbitrary hyperfields and $f\colon \HH_1 \rightarrow \HH_2$ a relatively algebraically closed homomorphism.
Then, for any polynomial $p \in \HH_1[X_1, \dots, X_n]$,
\[
V(f_*(p)) = \SetOf{(f(a_1), \dots, f(a_n)) \in \HH_2^n}{\ba \in V(p)} \, .
\]
\end{theorem*}

Prior to this work, all examples of RAC maps were either of the form $\HH \rightarrow \TT$ or $\HH \rightarrow \KK$. 
To motivate generalising Kapranov's theorem, we exhibit a number of new RAC maps.
We first show that any homomorphism from an algebraically closed field to a \emph{stringent hyperfield}, i.e. addition is multivalued only when summing additive inverses, is necessarily RAC (Corollary~\ref{cor:stringent+rac}).
Coupled with Bowler and Su's classification of stringent hyperfields~\cite{BowlerSu:21}, this covers all valuations and a new family of maps that we call \emph{fine valuations}, one class of enriched valuations.
We also show that RAC maps are closed under two operations, tropical extension and quotient (Propositions~\ref{prop:extension+rac} and~\ref{prop:RAC-factor-RAC}).
In particular, the former generalises the non-trivial RAC map in~\cite{Maxwell:21} from the tropical complex hyperfield to $\TT$.

One might hope to extend Kapranov's theorem to the Fundamental theorem for all RAC maps directly by considering ideals rather than single polynomials.
However, this is rather more subtle for hyperfields than over fields.
Attempting to imbue the set of polynomials $\HH[X_1, \dots, X_n]$ over a hyperfield $\HH$ with additional algebraic structure results in a lot of pathological behaviour.
In particular, it is not immediately clear what the correct definition of a polynomial ideal over $\HH$ should be.
As such, we restrict ourselves to RAC homomorphisms from fields where we have a well-defined notion of polynomial ideals.

\begin{theorem*}[\ref{thm:fund+thm+proj}]
    Let $f: \FF \rightarrow \HH$ be a relatively algebraically closed homomorphism from an algebraically closed field $\FF$ to a hyperfield $\HH$.
    Then, for any ideal $I \subseteq \FF[X_1 \dots X_n]$, 
    \[
    V(f_*(I)) = \SetOf{(f(a_1), \dots, f(a_n)) \in \HH^n}{\ba \in V(I)} \, .
    \]
\end{theorem*}

Specialising $\HH = \TT$ and $f$ a valuation recovers the Fundamental theorem of tropical geometry.
We can also consider this theorem for our new family of RAC maps, namely fine valuations.
The prototypical example of such a map is an extension of the valuation map on the field of Puiseux series that records the entire leading term rather than just the leading exponent:
\begin{align*}
  \fval\colon\pseries{\CC} &\rightarrow (\CC^\times \times \QQ) \cup \{\infty\} \, , & \\
  \sum_{i=0}^\infty c_it^{a_i} &\mapsto c_0t^{a_0} \, , & a_0 < a_1 < a_2 < \dots \in \QQ \, .
\end{align*}
With the fundamental theorem in hand, this motivates the definition of \emph{fine tropical varieties} (Definition~\ref{def:fine+trop+variety}) as a refinement of tropical varieties.
These are initially defined as the images of algebraic varieties in the fine valuation map, but the previous theorem allows us to view them also as the solution set of polynomials over the `fine tropical hyperfield'.
Utilising this perspective, we give a structure theorem for fine tropical varieties, describing them as tropical varieties with algebraic varieties associated to each polyhedral cell determined by their initial ideals (Theorem~\ref{thm:fine+grobner}).
We end by motivating their study further by exhibiting applications to stable intersection of tropical varieties, as well as links to polyhedral homotopy continuation where they have already been studied under another guise.

\subsection{Related work}

Hyperfields are related to many other algebraic objects, and many applicable results may be stated in these terms.
Baker and Bowler's work developing matroids over hyperfields~\cite{Baker+Bowler:18} was quickly generalised to other `partial hyperstructures' including partial fields, pastures and tracts~\cite{Baker+Bowler:19}.
These also encompass the fuzzy rings introduced by Dress~\cite{Dress:86}, whose link with hyperfields was explicitly given in~\cite{Gian+Jun+Lor:17}.
All of these objects belong to a larger category of algebraic objects called blueprints~\cite{Lorscheid:12,Lorscheid:22}, fundamental objects within $\FF_1$-geometry.
Hyperfields are also closely related to semirings in a number of ways.
The tropical hyperfield $\TT$ (and other tropical extensions) was first studied as (extensions of) the tropical semifield.
Moreover, we can `lift' the tropical hyperfield to a power semiring whose elements are subsets of $\TT$, making hyperaddition a singlevalued operation: this object is isomorphic to Izhakian's extended tropical arithmetic~\cite{IZH:09}.
Lifting more general hyperfields to semirings leads to the notion of semiring systems \cite{Row:21,AGR:22,AGT:23}. 

Alternative hyperfields have been considered for the study of tropical geometry, although often in the language of semirings.
The most studied is the signed tropical hyperfield, first introduced to enrich $(\max, +)$-linear algebra~\cite{MaxPlus:90,Gaubert:92}.
It has proved useful for studying semialgebraic sets over real valued fields~\cite{AllamigeonGaubertSkomra:20,JSY:22}, where the valuation map is enriched to record the sign of an element.
Moreover, it has close ties to real algebraic geometry via Mazlov dequantization and Viro's patchworking method~\cite{Viro:84}. 
Applications in enumerative tropical geometry motivated the introduction of a complexified valuation map, which records both the valuation and the phase of an element.
This was used by Mikhalkin to enumerate curves in toric surfaces~\cite{Mikhalkin:05}, and motivated a complex analogue of Mazlov dequantization introduction by Viro~\cite{Viro:11}.
Other tropical extensions of semirings have been studied with regards to tropical linear algebra~\cite{AGG:09, AGG:14}.
Moreover, non-tropical hyperfields also arise naturally: one can view the theory of amoebas and coamoebas through the hyperfield lens~\cite{Purbhoo:08,Forsgaard:15,Nisse+Passare:17,deWolff:17,Forsgaard:21}, as images of varieties in the triangle and phase hyperfields respectively.

Our approach to geometry of hyperfields requires studying roots of polynomials over hyperfields.
First discussed by Viro with respect to tropical geometry~\cite{Viro:11}, the first systematic study was by Baker and Lorscheid~\cite{Baker+Lorscheid:18} who introduced the notion of the multiplicity of a root, and unified Descartes' rule of signs and Newton's polygon rule in the framework of hyperfields.
This inspired a flurry of progress in the study of roots and factorizations of polynomials over hyperfields~\cite{Agu+Lor:21,Gunn:22,GrossGunn:23,AGT:23}.
An alternative scheme-theoretic approach to geometry over hyperfields was developed in \cite{JJun,Jun:21}.


\subsection{Structure of the paper}

The structure of the paper is as follows.
In Section~\ref{sec:hyperfields}, we recall the necessary preliminaries of hyperfields.
We give a number of key examples and constructions, including factor hyperfields and tropical extensions of hyperfields.
We then recall homomorphisms of hyperfields and introduce enriched valuations, framed as homomorphisms from fields to tropical extensions.
We finally recall the necessary details to define affine and projective (pre)varieties over hyperfields, including polynomials and the deficiencies with polynomial sets.

Section~\ref{sec:rac} is dedicated to studying relatively algebraically closed maps.
We first motivate their study by observing deficiencies with algebraically closed hyperfields that do not arise for fields.
We then deduce that homomorphisms from algebraically closed fields to stringent hyperfields are necessarily RAC, giving rise to a new family of RAC maps, namely fine valuations.
We also show that RAC maps are closed under taking tropical extensions and certain quotients.

Section~\ref{sec:main+thms} is where we prove our main theorems, the generalisations of Kapranov's theorem and the Fundamental theorem (Theorems~\ref{thm:Kapranov} and~\ref{thm:fund+thm+proj}).
As an application of these theorems, we define fine tropical varieties in Section~\ref{sec:fine+tropical+varieties} and prove a structure theorem for them (Theorem~\ref{thm:fine+grobner}).
We then demonstrate two potential applications of fine tropical varieties, namely stable intersection and polyhedral homotopies.
We end with some unresolved questions and avenues for further study in Section~\ref{Sec:further-qs}.

\subsection*{Acknowledgements}
We thank Jeff Giansiracusa for useful conversations and feedback on an early draft of the article, and an anonymous referee for helpful and detailed comments.
We also thank Farhad Babaee and Yue Ren for helpful conversations, and Trevor Gunn for communicating related work.
JM and BS were both supported by the Heilbronn Institute for Mathematical Research.

\section{Hyperfields} \label{sec:hyperfields}
In this section, we recall the necessary preliminaries of hyperfields.
This includes key definitions, examples and constructions of hyperfields and the maps between them.
We also discuss how to understand polynomials over hyperfields and establish a framework for their varieties.
For an alternative treatment of hyperfields, see the articles of Viro~\cite{Viro:10, Viro:11}.

Given a set $\HH$, we denote the set of non-empty subsets of $\HH$ by $\cP^*(\HH)$. 
A \emph{hyperoperation} on $\HH$ is a map $\boxplus\colon\HH\times\HH\rightarrow \cP^*(\HH)$ that sends two elements $a,b \in \HH$ to a non-empty subset $a \boxplus b$ of $\HH$.
This map can be extended to strings of elements, or sums of subsets, via:
\begin{align*}
    a_1 \boxplus a_2 \boxplus \dots \boxplus a_k &:= \bigcup_{a' \in a_2 \boxplus \dots \boxplus a_k} a_1 \boxplus a' \quad , \quad a_i \in \HH \, , \\
    A \boxplus B &:= \bigcup_{a \in A, b \in B} a \boxplus b \quad , \quad A, B \subseteq \HH \, .
\end{align*}
For ease of notation, we shall refrain from using set brackets for singleton sets.
The hyperoperation $\boxplus$ is called \emph{associative} if it satisfies
\[
(a \boxplus b) \boxplus c = a \boxplus (b \boxplus c) \quad , \quad \forall a,b,c \in \HH \, , 
\]
and is called \emph{commutative} if it satisfies
\[
a \boxplus b = b \boxplus a \quad \forall a,b \in \HH \, .
\]
\begin{definition}

A (canonical) \emph{hypergroup} is a tuple $(\HH, \boxplus, \0)$ where $\boxplus$ is an associative and commutative hyperoperation satisfying:
\begin{itemize}
    \item (Identity) $\0 \boxplus a = a$ for all $a \in \HH$,
    \item (Inverses) For all $a \in \HH$, there exists a unique $(-a) \in \HH$ such that $\0 \in a \boxplus -a$,
    \item (Reversibility) $a \in b \boxplus c$ if and only if $c \in a \boxplus (-b)$.
\end{itemize}
A \emph{hyperring} is a tuple $(\HH, \boxplus, \odot, \0, \1)$, where $\boxplus$ is hyperaddition and $\odot$ is multiplication, satisfying,
\begin{itemize}
\item $(\HH, \boxplus, \0)$ is a canonical hypergroup,
\item $(\HH^\times, \odot, \1)$ is a commutative monoid, where $\HH^\times := \HH \setminus \0\,$,
\item $a\odot(b \boxplus c) = a\odot b \boxplus a \odot c$ for all $a,b,c \in \HH\quad$ (where $a\odot(b \boxplus c) = \bigcup_{d \in b\boxplus c} a \odot d\,$)
\item $\0 \odot a = \0$ for all $a \in \HH$. 
\end{itemize}

A \emph{hyperfield} is a hyperring such that $(\HH, \odot, \1)$ is an abelian group, i.e. for every $a \in \HH^{\times}$, there exists a unique element $a^{-1}$ such that $a \odot a^{-1} = \1$.
\end{definition}


\begin{example}
We briefly recall some natural examples of hyperfields.
Many of these examples are linked as we shall see later; for a diagram and description of their relations, see~\cite{AndersonDavis:19}.
    \begin{itemize}
        \item $\FF$ - An ordinary field $\FF$ can be trivially viewed as a hyperfield, where the hyperaddition is defined as $a \boxplus b = a+b$.
        \item $\KK$ - The \emph{Krasner hyperfield} is the set $\KK := \{\0,\1\}$ with the standard multiplication and hyperaddition defined as
\[
\0 \boxplus \0 = \0 \quad , \quad \0 \boxplus \1 = \1 \boxplus \0 = \1 \quad , \quad \1 \boxplus \1 = \{ \0 , \1\} \, .
\]
        \item $\SS$ - The \emph{sign hyperfield} is the set $\SS :=\{-\1,\0,\1\}$ with standard multiplication and hyperaddition defined as
\[
\0 \boxplus a = a \quad , \quad a \boxplus a = a \quad , \quad \1 \boxplus -\1 = \{-\1,\0,\1\} \quad \forall a \in \SS .
\]
        \item $\TT$ - The \emph{(min-)tropical hyperfield} is the set $\TT := \RR \cup \{\infty\}$ with hyperfield operations
\begin{align*}
g \boxplus h &= 
\begin{cases}
\min(g,h)  &\text{if} \quad g \neq h\\
\{ g' \in \TT \, | \, g' \ge g\} &\text{if} \quad g = h
\end{cases} \, , \\
g \odot h &= g + h \, .
\end{align*} 
The additive and multiplicative identity elements are $\0 = \infty$ and $\1 = 0$.

The $\max$-tropical hyperfield is the isomorphic hyperfield obtained by replacing $\min$ with $\max$ and $\infty$ with $-\infty$.
As the $\min$ convention is consistent with the theory of valuations, we will use this throughout unless stated otherwise.
        \item $\PP$ - The \emph{phase hyperfield} $\PP = \SetOf{z \in \CC}{|z| = 1} \cup \{0\}$ is the hyperfield with operations
\begin{align*}
    z_1 \odot z_2 &= z_1\cdot z_2 \\
    z_1 \boxplus z_2 &= \begin{cases}
        z_1 & z_1 = z_2 \\
        \{z_1, 0, -z_1\} & z_2 = -z_1 \\
        \text{shortest open arc between } z_1, z_2 & \text{otherwise}
    \end{cases}
\end{align*}
        \item $\Phi$ - The \emph{tropical phase hyperfield} $\Phi = \SetOf{z \in \CC}{|z| = 1} \cup \{0\}$ is the hyperfield with operations
\begin{align*}
    z_1 \odot z_2 &= z_1\cdot z_2 \\
    z_1 \boxplus z_2 &= \begin{cases}
        z_1 & z_1 = z_2 \\
         \SetOf{z \in \CC}{|z| = 1} \cup \{0\} & z_2 = -z_1 \\
        \text{shortest closed arc between } z_1, z_2 & \text{otherwise}
    \end{cases}
\end{align*}
        \item $\TT\RR$ - The \emph{signed tropical hyperfield} is the set $\TT\RR:= (\{\pm 1\} \times \RR) \cup \{\infty\}$, where $\{-1\} \times \RR$ is a `negative' copy of the tropical numbers.
Its hyperfield operations are
\begin{align*}
    (s_1,g_1) \boxplus (s_2,g_2) &=
    \begin{cases}
        (s_1,g_1), & \text{if} \, g_1 < g_2,\\
        (s_2,g_2), & \text{if} \, g_2 < g_1,\\
        (s_1,g_1), & \text{if} \, s_1 = s_2, \, \text{and} \, g_1 = g_2,\\
        \{(\pm 1, h) \mid h \geq g_1\} \cup \{\infty\}, & \text{if} \, s_1 = -s_2, \, \text{and} \, g_1 = g_2 ,
    \end{cases} \\
    (s_1,g_1) \odot (s_2,g_2) &= (s_1 \cdot s_2, g_1 + g_2) \, .
\end{align*}
The additive and multiplicative identity elements are $\0 = \infty$ and $\1 = (1,0)$.
As with $\TT$, we can obtain an isomorphic hyperfield by replacing $\min$ with $\max$ and $\infty$ with $-\infty$. 

Note that there is an alternative description of $\TT\RR$ as the \emph{real tropical hyperfield} with underlying set $\RR$.
The isomorphism between these spaces is given by
\[
\TT\RR \rightarrow \RR \quad , \quad (s,g) \mapsto s\cdot\exp(-g) \, ,
\]
where the sign change switches between the min and max convention. 
        \item $\TT\CC$ - The \emph{tropical complex hyperfield} $\TT\CC$ has the complex numbers $\CC$ as its underlying set, with multiplication given by standard complex multiplication and hyperaddition defined as:
\begin{equation} 
z \boxplus w = \nonumber
\begin{cases}
z & \quad \text{if}\quad |z| > |w|,\\
w & \quad \text{if} \quad |w| >|z|,\\
\text{Shortest closed arc connecting}\; z \; \text{and} \; w\; \text{with radius}\; |z|, & \quad \text{if} \quad |z|=|w|, \, z\neq -w,\\
\{ c \in \CC \mid |c| \le |z| \}, & \quad \text{if} \quad  z=-w.
\end{cases}
\end{equation}
    \end{itemize}
\end{example}

A very general method for constructing hyperfields is to quotient a field by a subgroup of its units.
\begin{definition}\label{def:factor}
A \emph{factor hyperfield} is a hyperfield $\HH = \FF/U$ arising as the quotient of a field $(\FF,+,\cdot)$ by a multiplicative subgroup $U \subseteq \FF^{\times}$.
The elements of $\HH$ are cosets $\overbar{a} := a \cdot U = \{ a \cdot u \mid u \in U\}$, and the operations are inherited from the field operations:
\begin{align*}
\overbar{a} \boxplus \overbar{b} &= \{ \overbar{c} \mid c \in \overbar{a}+\overbar{b}\} \, ,\\
\overbar{a} \odot \overbar{b} &= \overbar{a \cdot b}
\end{align*}
\end{definition}

\begin{example}\label{ex:factor+hyperfields}
Many of the examples we have seen can be realised as factor hyperfields.
For instance, we can realise the following as the quotients
\[
\FF \cong \FF/\1 \quad , \quad \KK \cong \FF/\FF^{\times} \quad , \quad \SS \cong \RR/\RR_{>0}  \quad , \quad \PP \cong \CC/\RR_{>0}
\] 
where $\FF\neq \FF_2$ is any field. 
\end{example}
\begin{remark}
Given some arbitrary hyperfield $\HH$ and multiplicative subgroup $U \subseteq \HH^\times$, we can define the quotient $\HH/U$ analogously as in Definition~\ref{def:factor}.
The proof that this construction gives a hyperfield is identical to the field case given in~\cite{Kra:83}. 
\end{remark}

Next we will present the definition of a \emph{tropical extension} of a hyperfield.
\begin{definition}\label{def:tropical+extension}
Let $\HH$ be a hyperfield and $(\Gamma,+)$ an (additive) ordered abelian group.
The \emph{tropical extension} of $\HH$ by $\Gamma$ is the set
\[
\HH\rtimes \Gamma = \{(c,g) \mid c \in \HH^\times \, , \, g \in \Gamma\} \cup \{\0\}
\]
where multiplication is given by $(c_1,g_1)\odot(c_2, g_2) = (c_1\odot_\HH c_2 , g_1+g_2)$ and addition given by
\[
(c_1 , g_1) \boxplus (c_2 , g_2) =
\begin{cases}
    (c_1 , g_1) & g_1 < g_2 \\
    (c_2 , g_2) & g_2 < g_1 \\
    \{(c,g) \mid c \in c_1 \boxplus_{\HH} c_2\} & g_1 = g_2 = g \, , \, \0_\HH \notin c_1 \boxplus c_2\\
    \{(c,g) \mid c \in (c_1 \boxplus_{\HH} c_2) \setminus \0_\HH\}\cup \{(b,h) \mid h > g \, , \, b \in \HH^\times\} \cup \{\0\} & g_1 = g_2 = g \, , \, \0_\HH \in c_1 \boxplus c_2\\
\end{cases}
\]
\end{definition}

\begin{example}\label{ex:trop-hyp-trop-ext}
    The tropical hyperfield can be realised as the tropical extension $\TT \cong \KK\rtimes \RR$ where $(\RR, +)$ is viewed as an ordered abelian group.
    We can extend this correspondence to the \emph{rank k tropical hyperfield} $\TT^{(k)} \cong \KK\rtimes \RR^{(k)}$ where $(\RR^{(k)}, +)$ is $k$-tuples of reals ordered lexicographically.
\end{example}
\begin{example}\label{ex:sgn-trop-hyp-trop-ext}
    The signed tropical hyperfield can be viewed as the tropical extension $\TT\RR \cong \SS\rtimes \RR$.
\end{example}

\begin{example}\label{ex:tc+trop+ext}
    The tropical complex hyperfield can be viewed as the tropical extension $\TT\CC \cong \Phi\rtimes \RR$.
    As the underlying set of $\TT\CC$ is the complex numbers, this isomorphism can be explicitly written as
    \begin{align*}
    \TT\CC &\rightarrow \Phi \rtimes \RR \\
    z &\mapsto 
    \begin{cases} 
    (\ph(z), -\log(|z|)) & z \neq 0 \\ 
    \0 & z = 0    
    \end{cases} \, ,
    \end{align*}
    where $\ph(z) = z/|z|$ is the \emph{phase map}.
\end{example}

As a brief application of this construction, we recall Bowler and Su's classification of \emph{stringent} hyperfields.

\begin{definition}
    A hyperfield $\HH$ is \emph{stringent} if $a \boxplus b$ is a (non-singleton) set if and only if $a = -b$.
\end{definition}

\begin{theorem}\cite{BowlerSu:21}\label{thn:stringent+classification}
    Let $\HH$ be a stringent hyperfield.
    Then $\HH$ is isomorphic to either $\KK\rtimes \Gamma, \SS\rtimes \Gamma$ or $\FF\rtimes \Gamma$ for some ordered abelian group $\Gamma$.
\end{theorem}

\begin{example}
Observe that $\FF, \KK, \SS, \TT$ and $\TT\RR$ are all stringent.
This can be seen directly, as the only sets we can obtain in each case are those coming from the sum of an element with its inverse:
\begin{align*}
\KK &\colon \1 \boxplus \1 = \{\0,\1\} & \SS &\colon \1 \boxplus -\1 = \{-\1, \0, \1\} \\
\TT &\colon g \boxplus g = \{h \in \TT \mid h \geq g\}  & \TT\RR &\colon (1,g) \boxplus (-1,g) = \{(\pm 1, h) \mid h \geq g\} \cup \{\infty\}
\end{align*} 
However, this can also be seen from Theorem~\ref{thn:stringent+classification}: $\KK, \SS$ and $\FF$ can be trivially written as a tropical extension by the trivial group, and $\TT$ and $\TT\RR$ are tropical extensions of $\KK$ and $\SS$ by $\RR$ respectively.
\end{example}


\begin{remark}
Outside of hyperfields, the tropical extension construction has been defined and studied for a number of algebraic structures with both single and multivalued addition.
For single-valued structures, it has been used to study idempotent semirings~\cite{AGG:09, AGG:14, IzhakianKnebuschRowen:14} and more recently for \emph{semiring systems} as a bridge between semirings and hyperstructures~\cite{Row:21,AGR:22,AGT:23}.
For multivalued structures, it can be defined for a number of generalisations of hyperfields including \emph{idylls}, which were investigated as a way to study multiplicities of roots of polynomials~\cite{Gunn:22}.

We note that this construction is also sometimes refereed to as \emph{layering} rather than extension in the literature.
\end{remark}

\subsection{Homomorphisms and enriched valuations}\label{subsec:homo+val}
We will now recall the definition of a hyperfield homomorphism, connecting this with tropical extensions and valuations, and construct several key examples.

\begin{definition}
A map $f\colon\HH_1\rightarrow\HH_2$ is a \emph{hyperfield homomorphism} if it satisfies
\begin{align*}
f(a \boxplus_1 b) &\subseteq f(a) \boxplus_2 f(b) \, , & f(a \odot_1 b) &= f(a) \odot_2 f(b) \, , \\
f(\0_1) &= \0_2 \, , & f(\1_1) &= \1_2 \, .
\end{align*}
An \emph{isomorphism} of hyperfields is a bijective homomorphism whose inverse map is also a homomorphism.
\end{definition}
As with field homomorphisms, it is also straightforward to show that $\0_1$ is the unique element that gets sent to $\0_2$ by $f$.
We can also similarly deduce that $f(a)^{-1} = f(a^{-1})$ and $-f(a) = f(-a)$ for all $a \in \HH_1$.

\begin{remark}
We will often extend $f$ to a map between spaces $f \colon \HH_1^n \rightarrow \HH_2^n$, obtained by applying $f$ coordinatewise.
Similarly, if $S \subseteq \HH_1^n$ is a subset, then we let $f(S) \subseteq \HH_2^n$ be the set obtained by applying $f$ to each element of $S$. 
\end{remark}

\begin{example}
    Every hyperfield has a trivial homomorphism to the Krasner hyperfield, given by
\begin{align*}
    \omega:\HH &\rightarrow \KK \quad , \quad \omega(a) =
    \begin{cases}
        \1, \quad \text{if} \quad a\neq \0_{\HH}. \\
        \0, \quad \text{if} \quad a=\0_{\HH}.
    \end{cases}
\end{align*}
\end{example}

 \begin{example}\label{ex:factor+hom}
 Given a factor hyperfield $\HH=\FF/U$, there is a natural hyperfield homomorphism
 \begin{equation*}
     \tau:\FF \rightarrow \HH \quad , \quad a \mapsto \overbar{a} \, .
 \end{equation*}
 Given that all the hyperfields we have see are factor hyperfields, this gives many examples of hyperfield homomorphisms, e.g.
 \begin{align*} 
  \sgn:\RR &\rightarrow \SS \cong \RR/\RR_{>0} \, , & \ph:\CC &\rightarrow \PP \cong \CC/\RR_{>0} \, , \\
a &\mapsto 
 \begin{cases}
 \1 & \text{if} \quad  a \in \RR_{>0} \\
 -\1 & \text{if} \quad a \in \RR_{<0} \\
 \0 & \text{if} \quad a=0
 \end{cases} & 
 z &\mapsto 
\begin{cases}
\frac{z}{|z|} & \text{if} \quad z \in \CC\backslash\{0\} \\
 0 &\text{if} \quad z=0
 \end{cases}
 \end{align*}
\end{example}



  
As with tropical extensions of hyperfields, one can consider tropical extensions of homomorphisms.
  \begin{definition}\label{def:hom+extension} 
    Let $f\colon \HH_1 \rightarrow \HH_2$ be a map between two hyperfields.
    Given an ordered abelian group $\Gamma$, the \emph{tropical extension of $f$ by $\Gamma$} is the map
    \begin{align*}
    f^\Gamma\colon \HH_1\rtimes \Gamma &\rightarrow \HH_2\rtimes \Gamma \\
    (c,g) &\mapsto (f(c),g) \, .
    \end{align*}
    It is easy to see that if $f$ is a homomorphism, then so is $f^\Gamma$.
  \end{definition}
  
\begin{example}\label{ex:trop+complex+map}
    Consider the homomorphism from the tropical complex hyperfield $\TT\CC$ to the tropical hyperfield:
    \begin{align*}
    \eta \colon \TT\CC &\rightarrow \TT \\
    z &\mapsto -\log(|z|) \, .
    \end{align*}
    This is the key example studied in~\cite{Maxwell:21}, as it `preserves' roots of polynomials; we shall discuss this in detail in Section~\ref{sec:rac}.
    We can give an alternative description of this map in the language of tropical extensions.
    As both $\TT\CC \cong \Phi\rtimes \RR$ and $\TT \cong \KK\rtimes \RR$ are tropical extensions, the map can be rephrased as
    \begin{align*}
    \omega^\Gamma \colon \Phi\rtimes \RR &\rightarrow \KK\rtimes \RR \\
    (\theta, g) &\mapsto (\1,g) \, ,
    \end{align*}
    where $\theta = \ph(z)$ and $g = -\log(|z|)$ given by the isomorphism in Example~\ref{ex:tc+trop+ext}.
In particular, $\eta$ is the extension of the trivial homomorphism $\omega\colon \Phi \rightarrow \KK$ by $\RR$.
  
 \end{example}
 
A crucial example of a hyperfield homomorphism will be those coming from valued fields.
As this motivates many of the applications of our theory, we recall the definition.

\begin{definition}
    A \emph{valuation} on a field $\FF$ is a surjective map $\val\colon \FF \rightarrow \Gamma \cup \{\infty\}$ to an ordered abelian group $(\Gamma,+)$ satisfying
    \begin{itemize}
        \item $\val(a) = \infty \Leftrightarrow a = 0$,
        \item $\val(ab) = \val(a) + \val(b)$,
        \item $\val(a+b) \geq \min(\val(a),\val(b))$, with equality if and only if $\val(a) \neq \val(b)$.
    \end{itemize}
    The pair $(\FF, \val)$ is called a \emph{valued field}.
    The group $\Gamma$ is called the \emph{value group} of $\FF$.
\end{definition}
Note that we can demand that valuations are surjective by restricting the range of non-surjective valuations to the value group.
This will be beneficial later, allowing us to sidestep certain topological concerns we would have to deal with otherwise.

\begin{example}\label{ex:puiseux+series} 
    A natural example of a valued field is the field of \emph{Puiseux series} over a field $\FF$.
    These are formal power series whose exponents are rational with a common denominator
    \[
    \pseries{\FF} = \bigcup_{n \geq 1} \FF(\!(t^{\frac{1}{n}})\!) = \SetOf{\sum_{i=0}^\infty c_it^{g_i}}{\begin{array}{c} c_i \in \FF \, , \, c_0 \neq 0 \\ g_i \in \QQ \text{ with common denominator }\\g_0 < g_1 < \dots < g_i < \dots \end{array} } \, ,
    \] 
    along with the zero element.
    If $\FF$ is an algebraically closed field, then $\pseries{\FF}$ is also algebraically closed: in fact, it is the algebraic closure of the field of Laurent series $\FF(\!(t)\!)$.

    The `first' term $c_0t^{g_0}$ is referred to as the \emph{leading term}, where $c_0$ is the \emph{leading coefficient} and $g_0$ is the \emph{leading exponent}.
    The valuation on $\pseries{\FF}$ maps zero to $\infty$, and a non-zero series to its leading exponent:
    \begin{align*}
    \val\colon \pseries{\FF} &\rightarrow \QQ \cup\{\infty\} \\ \sum_{i=0}^\infty c_it^{g_i} &\mapsto g_0 \, .
    \end{align*}
\end{example}

\begin{example}\label{ex:hahn+series}
    A more general example of a valued field is the field of \emph{Hahn series} with value group $\Gamma$ over a field $\FF$.
    These are formal power series whose exponents are elements of the ordered abelian group $\Gamma$
    \[
    \hseries{\FF}{\Gamma} = \SetOf{\sum_{g \in G} c_g t^g}{c_g \in \FF^\times \, , \, G \subseteq \Gamma \text{ well-ordered } } \, ,
    \]
    along with the zero element.
    The field $\hseries{\FF}{\Gamma}$ is algebraically closed if $\FF$ is algebraically closed and $\Gamma$ is \emph{divisible} i.e. for all $g \in \Gamma$ and $n \in \NN$, there exists $h \in \Gamma$ such that $g = n \cdot h$.
    
    We use the same terminology for leading term/coefficient/exponent as with Puiseux series: the condition that $G$ is well-ordered ensures it exists.
    The valuation on $\hseries{\FF}{\Gamma}$ behaves identically to the valuation on $\pseries{\FF}$, mapping zero to $\infty$ and a non-zero series to its leading exponent:
   \begin{align} \label{eqn:val-trop-ext}
       \val\colon \hseries{\FF}{\Gamma} &\rightarrow \Gamma \cup\{\infty\} \\ 
    \sum_{g \in G} c_g t^g &\mapsto \gamma = \min(g \mid g \in G) \, .\nonumber 
   \end{align} 
\end{example}

Valuations can be considered hyperfield homomorphisms in the following way.
By identifying $\Gamma \cup \{\infty\}$ with $\KK\rtimes \Gamma$, where $\0 = \infty$, we can instead consider $\val$ as a map to $\KK\rtimes \Gamma$.
The following proposition is easy to verify.

\begin{proposition}\label{prop:val}
    $(\FF, \val)$ is a valued field if and only if $\val\colon \FF \rightarrow \KK\rtimes \Gamma$ is a hyperfield homomorphism.
\end{proposition}
\begin{proof}
    It is clear the first two conditions of a valuation are hyperfield homomorphism properties.
    Moreover, it is straightforward to check $\val(1) = 0$.
    For the final condition, recall that
    \[
    \val(a) \boxplus \val(b) = \begin{cases}
        \min(\val(a), \val(b)) & \val(a) \neq \val(b) \\ \SetOf{h \in \Gamma \cup \{\infty\}}{h \geq \val(a)} & \val(a) = \val(b)
    \end{cases} \, .
    \]
    This implies that $\val(a+b) \in \val(a) \boxplus \val(b)$ is equivalent to the final valuation condition.
\end{proof}

Proposition~\ref{prop:val} shows that valuations are equivalent to homomorphisms from a field to the tropical extension $\KK\rtimes \Gamma$.
This motivates the study of \emph{enriched valuations}: homomorphisms $\FF \rightarrow \HH\rtimes \Gamma$ from a field to a tropical extension, where $\HH$ records additional information about elements over $\FF$.


\begin{example}\label{ex:sval}
    For an ordered field $\FF$ and ordered abelian group $\Gamma$, the valuation map~\eqref{eqn:val-trop-ext} can be enriched to the \emph{signed valuation}, which records the sign of a Hahn series:
\begin{align*}
\sval\colon \hseries{\FF}{\Gamma} & \rightarrow \SS\rtimes \Gamma ,\\ \sum_{g \in G} c_g t^g  &\mapsto (\sgn(c_{\gamma}), \gamma) \, , \quad \gamma = \min(g \mid g \in G) \, ,
\end{align*}
where $\sgn$ defined as in Example~\ref{ex:factor+hom}.
Setting the value group $\Gamma$ to $\RR$ recovers the signed valuation to the signed tropical hyperfield $\SS \rtimes \RR \cong \TT\RR$.
\end{example}

\begin{example}\label{ex:fine+val}
    Consider the field of Hahn series $\hseries{\FF}{\Gamma}$ over an arbitrary field.
    We can enrich the usual valuation map~\eqref{eqn:val-trop-ext} on $\hseries{\FF}{\Gamma}$ by defining the \emph{fine valuation} $\fval$ that remembers the coefficient of its leading term, i.e.
    \begin{align*}
        \fval\colon \hseries{\FF}{\Gamma} &\rightarrow \FF\rtimes \Gamma \\ \sum_{g \in G} c_g t^g &\mapsto (c_\gamma , \gamma)\, , \quad \gamma = \min(g \mid g \in G) \, .
    \end{align*}
As an application of our results, the tropical geometry of this valuation is discussed in Section~\ref{sec:fine+tropical+varieties}.
\end{example}

\begin{example}\label{ex:phval}
    Fixing $\FF = \CC$, the valuation map~\eqref{eqn:val-trop-ext} can be enriched to the \emph{phased valuation}, which records the phase or argument of a Hahn series:
\begin{align*}
\phval\colon \hseries{\CC}{\RR} & \rightarrow \PP\rtimes \Gamma ,\\ \sum_{g \in G} c_g t^g  &\mapsto (\ph(c_{\gamma}), \gamma)\, , \quad \gamma = \min(g \mid g \in G) \, ,
\end{align*}
where $\ph$ is the phase map as defined as in Example~\ref{ex:factor+hom}.
\end{example}

We note that all of these examples also hold over Puiseux series, where the value group is $\Gamma = \QQ$.
However, we also note that $\pseries{\FF} \subsetneq\hseries{\FF}{\QQ}$ as the latter has no common denominator condition.

We end this section by noting that enriched valuations allow us to construct these hyperfields as factor hyperfields.
Explicitly, in each case we can realise each hyperfield as the domain of the enriched valuation modulo the preimage of $\1$ in the hyperfield:
\begin{align*}
\TT &\cong \hseries{\FF}{\RR}/\val^{-1}(0) &
\TT\RR &\cong \hseries{\FF}{\RR}/\sval^{-1}((\1_\SS,0)) \\
\FF \rtimes \Gamma &\cong \hseries{\FF}{\Gamma}/\fval^{-1}((\1_\FF,0)) &
\PP \rtimes \Gamma &\cong \hseries{\CC}{\Gamma}/\phval^{-1}((\1_\PP,0)) \\
\end{align*}
The proof of these isomorphisms follows from the following lemma.

\begin{lemma}\label{lem:factor+from+hom}
Let $f\colon \FF \rightarrow \HH$ be a surjective homomorphism from a field $\FF$ to a hyperfield $\HH$ satisfying
\begin{equation} \label{cond:preimage+sums}
\forall \alpha, \beta \in \HH \, , \, \gamma \in \alpha \boxplus \beta \, , \, \exists a \in f^{-1}(\alpha) \, , \, b \in f^{-1}(\beta) \text{ such that } a+b \in f^{-1}(\gamma) \, .
\end{equation}
Then $\HH \cong \FF/f^{-1}(\1_\HH)$.
\end{lemma} 
\begin{proof}
Denote $U = f^{-1}(\1_\HH)$ and consider the map $f_U \colon \FF/U \rightarrow \HH$ that maps $\overbar{a} \mapsto f(a)$.
To see this map is well-defined, consider $a,b \in \FF$ such that $\overbar{a} =\overbar{b}$, i.e. there exists $u \in U$ such that $a = b \cdot u$.
Then
\[
f_U(\overbar{a}) = f(a) = f(b \cdot u) = f(b) \odot f(u) = f(b) \odot \1_\HH = f(b) = f_U(\overbar{b}) \, ,
\]
hence the map $f_U$ is independent of the choice of representative for the cosets.

We claim that $f_U$ defines an isomorphism.
It is straightforward to check that $f_U$ is a surjective homomorphism given that $f$ is.
To see that $f_U$ is injective, suppose that $f(a) = f(b)$.
Then
\[
f(a) \cdot f(b)^{-1} = f(a \cdot b^{-1}) = \1_\HH \, \Rightarrow \, a \cdot b^{-1} \in U \, 
\]
Hence there exists $u \in U$ such that $a = b \cdot u$, and so $\overbar{a} = \overbar{b}$.

As $f_U$ is a bijective morphism, there exists an inverse map $f_U^{-1}\colon \HH \rightarrow \FF/U$ given by $\alpha \mapsto \overbar{f^{-1}(\alpha)}$.
It remains to show this is also a homomorphism.
The only non-trivial condition is that it preserves sums.
Consider some $\gamma \in \alpha \boxplus \beta$, then $f^{-1}_U(\gamma) = \overbar{f^{-1}(\gamma)}$.
By the condition~\eqref{cond:preimage+sums}, there exists $a \in f^{-1}(\alpha)$ and $b\in f^{-1}(\beta)$ such that $a+b \in f^{-1}(\gamma)$.
Coupled with the injectivity of $f_U^{-1}$, this implies 
\[
\overbar{f^{-1}(\gamma)} = \overbar{a+b} \subseteq \overbar{a}+\overbar{b} \subseteq f_U^{-1}(\alpha) \boxplus f^{-1}_U(\beta) \, .
\]
Ranging over all $\gamma \in \alpha \boxplus \beta$ gives $f_U^{-1}(\alpha \boxplus \beta) \subseteq f_U^{-1}(\alpha) \boxplus f^{-1}_U(\beta)$.
\end{proof}

\begin{remark}
Without condition~\eqref{cond:preimage+sums}, Lemma~\ref{lem:factor+from+hom} gives a bijective homomorphism from $\FF/f^{-1}(\1_\HH)$ to $\HH$.
It is important to stress that this is not sufficient for an isomorphism of hyperfields, as the following example highlights.

Let $\WW = \{\1, \0, -\1\}$ be the \emph{weak hyperfield of signs}, with the same multiplication and addition as $\SS$ aside from
\[
\1 \boxplus \1 = -\1 \boxplus - \1 = \{\1, - \1\} \, .
\]
One can check that the sign map $\sgn\colon \RR \rightarrow \WW$ is also a surjective homomorphism, but $\RR/\sgn^{-1}(\1) \cong \SS \not\cong \WW$.
This is because $\sgn\colon \RR \rightarrow \WW$ does not satisfy~\eqref{cond:preimage+sums}: we have $-\1 \in \1 \boxplus \1$ but we cannot find positive reals $a, b$ such that $a+b$ is negative.
\end{remark}

\subsection{Polynomials over hyperfields} \label{sec:polynomial}
We review some key notions of polynomials over hyperfields.
Some care is needed, as we shall see that polynomials behave much worse over hyperfields than over fields.

\begin{definition}
The set of polynomials in $n$ variables over a hyperfield $\HH$ will be denoted $\HH[X_1 \, , \dots , \, X_n]$, where elements of this set are defined as
\begin{equation*}\label{hyperpolynomial}
p(X_1 \, , \dots , \, X_n) :=  \bigboxplus_{\bd \in D} c_\bd \odot X_1^{d_1} \odot \, \cdots \, \odot X_n^{d_n} \, ,
\end{equation*}
where $D \subset \ZZ_{\geq 0}^n$ finite and $c_\bd \in \HH^\times$.
Each polynomial defines a function from $\HH^n$ to $\cP^*(\HH)$ given by evaluation.

The set of \emph{Laurent polynomials} over $\HH$ will be denoted $\HH[X_1^\pm, \dots, X_d^\pm]$, whose elements are polynomials whose exponents $D \subset \ZZ^n$ may take negative values.
Each Laurent polynomial defines a function from $(\HH^\times)^n$ to $\cP^*(\HH)$, as $\0^{k}$ is undefined for negative values of $k$.

We will use multi-index notation $\bX^\bd = X_1^{d_1}\odot\cdots\odot X_n^{d_n}$ to denote (Laurent) monomials.
\end{definition}

\begin{definition}
Let $p = \bigboxplus_{\bd\in D} c_\bd \odot \bX^\bd \in \HH[X_1,\dots,X_n]$. An element $\ba = (a_1 \, , \dots , \, a_n)$ is a \emph{root} of the polynomial if $\0 \in p(\ba) =  \bigboxplus_{\bd\in D} c_\bd \odot \ba^\bd$.
We denote the set of roots of $p$ as
\[
V(p) : = \{ \ba \in \HH^n \mid \0 \in p(\ba)\} \, .
\]
Note that $V(p)$ can also be considered as the (affine) hypersurface defined by $p$.
We will expand on this in Section~\ref{sec:varieties}.
\end{definition}

Let $f:\HH_1 \rightarrow \HH_2$ be a hyperfield homomorphism. 
This induces a map of polynomials $f_* : \HH_1[X_1,\dots,X_n] \rightarrow \HH_2[X_1,\dots,X_n]$ defined as
\[
p = \bigboxplus\nolimits_1  c_\bd \odot_1 \bX^\bd \longmapsto f_*(p)  = \bigboxplus\nolimits_2 f(c_\bd) \odot_2 \bX^{\bd} \, .
\]
We call $f_*(p)$ the \emph{push-forward} of $p$.
By properties of hyperfield homomorphisms, we observe that
\begin{equation*}\label{eq:pushforward+hom}
f(p(\ba)) = f\left(\bigboxplus\nolimits_1  c_\bd \odot_1 \ba^\bd\right) \subseteq \bigboxplus\nolimits_2 f(c_\bd) \odot_2 (f(\ba))^\bd = f_*(p)(f(\ba)) \, .
\end{equation*}
This gives the following result as a direct consequence.

\begin{lemma}\cite{Maxwell:21} \label{pushforwardgeneral}
Let $f: \HH_1 \rightarrow \HH_2$ be a hyperfield homomorphism. 
For $p \in \HH_1[X_1 \, , \dots , \, X_n]$, 
\[
f(V(p)) \subseteq V(f_*(p)) \, .
\]
\end{lemma}
For general hyperfield homomorphisms, this containment is strict.
However, for a particularly nice class of homomorphisms we will be able to deduce an equality in Lemma~\ref{pushforwardgeneral}.
These will be the class of relatively algebraically closed homomorphisms and will be the focus of Section~\ref{sec:rac}.

\subsection{Hyperfield varieties} \label{sec:varieties}

We briefly touched on the notion of an affine hypersurface over $\HH$ in the previous section.
More general varieties are trickier as polynomial ideals over hyperfields have a number of pitfalls that we document at the end of this section. 
As such, we must restrict ourselves to images of polynomial ideals over fields.
It will be useful for us to consider various notions of varieties over a hyperfield.

\paragraph{Affine varieties}
Given a polynomial $p \in \HH[X_1, \dots, X_n]$, its associated \emph{affine hypersurface} is
\[
V(p) := \SetOf{\ba \in \HH^n}{\0 \in p(\ba)} \, .
\]
Given a set of polynomials $J \subseteq \HH[X_1, \dots, X_n]$, its associated \emph{affine prevariety} is
\[
V(J) := \SetOf{\ba \in \HH^n}{\ba \in V(p) \text{ for all } p \in J} = \bigcap_{p \in J} V(p) \, .
\]
If there exists some ideal $I \subseteq \FF[X_1,\dots,X_n]$ over a field $\FF$ and homomorphism $f\colon \FF \rightarrow \HH$ such that $J = f_*(I) := \SetOf{f_*(p)}{p \in I}$, we will call $V(J)$ an \emph{affine variety}.

\begin{remark}
We briefly note that there is a well-defined notion of a tropical variety that does not depend on any underlying ideal over a field.
This differs from the notion above, where the varieties $V(J)$ such that $J = \val_*(I)$ for some ideal $I$ over a field $\FF$ would be known as \emph{realisable} tropical varieties.
For general hyperfields, we do not have such intrinsic notions of varieties yet and so restrict ourselves to those that are `realisable'.
\end{remark}


\paragraph{Projective varieties}
As multiplication is single-valued, projective space and projective varieties over a hyperfield are defined much the same as over fields.
As such, we shall skip over proofs of standard facts and refer to~\cite{CLO}, where proofs can be easily adapted to hyperfields.
We define \emph{$n$-dimensional projective space} over $\HH$ as
\[
\PP^n(\HH) = \faktor{\HH^{n+1} \setminus \{\zerovec\}}{\sim} \, ,
\]
where $\sim$ identifies scalar multiples, i.e. $\ba \sim \lambda \odot \ba$ for all $\lambda \in \HH^\times$.


A polynomial $p \in \HH[X_0, \dots, X_n]$ is \emph{homogeneous} if each monomial has the same degree.
If $p$ is homogeneous, then $\ba \in \HH^{n+1}$ is a root if and only if $\lambda \odot \ba$ is a root.
As such, we can consider roots of homogeneous polynomials as hypersurfaces in projective space.
Given a homogeneous polynomial $p \in \HH[X_0, X_1, \dots, X_n]$, its associated \emph{projective hypersurface} is
\begin{align} \label{eq:proj+hyp}
PV(p) :=& \SetOf{\ba \in \PP^n(\HH)}{\0 \in p(\ba)} = \faktor{V(p) \setminus \{\zerovec\}}{\sim} \, .
\end{align}

Given a set of homogeneous polynomials $J \subseteq \HH[X_0, X_1, \dots, X_n]$, its associated \emph{projective prevariety} is
\begin{align*}
PV(J) :=& \bigcap_{p \in J} PV(p) = \faktor{V(J) \setminus \{\zerovec\}}{\sim} \, .
\end{align*}
If there exists an ideal $I \subseteq \FF[X_0,X_1, \dots,X_n]$ over a field $\FF$ and homomorphism $f\colon \FF \rightarrow \HH$ such that $J = f_*(I)$, we will call $PV(J)$ a \emph{projective variety}.

\paragraph{Torus subvarieties}
Tropical geometry generally focuses on subvarieties of the torus, as these give rise to polyhedral complexes in $\RR^n$.
As such, we briefly recall the setup for this family of varieties also.
Recall the \emph{$n$-dimensional torus} over $\HH$ is $(\HH^\times)^n$.
Given a Laurent polynomial $p \in \HH[X_1^\pm, \dots, X_n^\pm]$, its associated \emph{torus hypersurface} is
\[
V^\times(p) := \SetOf{\ba \in (\HH^\times)^n}{\0 \in p(\ba)} \, .
\]
Given a set of Laurent polynomials $J \subseteq \HH[X_1^\pm, \dots, X_n^\pm]$, its associated \emph{torus prevariety} is
\[
V^\times(J) := \SetOf{\ba \in (\HH^\times)^n}{\ba \in V^\times(p) \text{ for all } p \in J} = \bigcap_{p \in J} V^\times(p) \, .
\]
If $J = f_*(I)$ for some Laurent ideal $I \subseteq \FF[X_1^\pm, \dots,X_n^\pm]$ over a field $\FF$ and homomorphism $f\colon \FF \rightarrow \HH$, we will call $V^\times(J)$ a \emph{torus subvariety}.

We will generally prove results for one of these three families of varieties and then deduce the analogous result for the other families.
This requires a dictionary for how to move between these families.
This is very similar to the situation over fields, therefore we shall state results without proof.

Affine space, projective space and the torus are related by the following inclusions
\[
\begin{array}{ccccc}
(\HH^\times)^n & \overset{i}{\hookrightarrow} & \HH^n & \overset{j}\hookrightarrow & \PP^n(\HH) \\
(a_1, \dots a_n) & \mapsto & (a_1, \dots a_n) &\mapsto & [\1 : a_1 : \dots : a_n]
\end{array} \, .
\]
The first inclusion is canonical, and so we will drop the $i$.
We denote $U = j(\HH^n)$ and $U^\times = j((\HH^\times)^n)$ as the affine and torus chart in $\PP^n(\HH)$ respectively.


\begin{lemma}\label{lem:aff+as+proj}
    Let $J \subseteq \HH[X_1, \dots, X_n]$ be a collection of polynomials.
    The affine prevariety $V(J) \subseteq \HH^n$ can be identified with the projective prevariety $PV(\overbar{J}) \subseteq \PP^n(\HH)$ restricted to the affine chart $U$, where $\overbar{J}$ is the homogenization of $J$: 
    \[
    \overbar{J} := \SetOf{\bigboxplus_{d \in D} c_{d}\odot X_0^{\deg(p) - \sum_{i=1}^n d_i} \odot X_1^{d_1} \odot \cdots \odot X_n^{d_n}}{p = \bigboxplus_{\bd \in D} c_\bd\odot \bX^{\bd} \in J} \subseteq \HH[X_0,X_1, \dots, X_n] \, .
    \]
\end{lemma}


\begin{lemma}\label{lem:torus+as+proj}
    Let $J \subseteq \HH[X_1^{\pm}, \dots, X_n^{\pm}]$ be a collection of Laurent polynomials.
    Then $V^\times(J) = V(J_{\aff}) \cap (\HH^\times)^n$ where $J_{\aff}$ is the affinization of $J$:
    \[
    J_{\aff} := \SetOf{\bigboxplus_{\bd \in D} c_\bd\odot X_1^{d_1+e_1} \odot \cdots \odot X_n^{d_n + e_n}}{\bigboxplus_{\bd \in D} c_\bd\odot \bX^{\bd} \in J \, , \, e_i = \min_{d \in D}(d_i, 0)} \subseteq \HH[X_1, \dots, X_n] \, .
    \]
    Moreover, the torus prevariety $V^\times(J) \subseteq (\HH^\times)^n$ can be identified with the projective prevariety $PV(\overline{J_{\aff}})$ restricted to the torus chart $U^\times$.
\end{lemma}

We end this section with a compilation of issues with polynomial rings and ideals over hyperfields for algebraic geometry, and why we restrict ourselves to polynomial ideals coming from fields.
This is not an exhaustive list: see \cite[Section 4.1]{JJun} for additional subtleties, as well as an alternative approach to hyperfield geometry.

\begin{remark}[Polynomial multiplication is multi-valued]\label{rem:hyperpoly}
    The set of polynomials over a field $\FF[X_1,\dots,X_n]$ has the structure of a ring.
    However, the set of polynomials over a general hyperfield $\HH[X_1,\dots,X_n]$ is not a hyperring, as multiplication is not single-valued.
  Instead, both addition and multiplication are multi-valued, defined as
      \begin{align*}
     \left(\bigboxplus_{\bd \in D} a_\bd \odot \bX^\bd\right) \boxplus \left(\bigboxplus_{\bd' \in D'} b_{\bd'} \odot \bX^{\bd'}\right) &= \SetOf{\bigboxplus_{\be \in D \cup D'} c_\be \odot X^\be}{c_\be \in a_\be \boxplus b_{\be}} \, ,  \nonumber\\
    \left(\bigboxplus_{\bd \in D} a_\bd \odot \bX^\bd\right) \boxdot \left(\bigboxplus_{\bd' \in D'} b_{\bd'} \odot \bX^{\bd'}\right) &= \SetOf{\bigboxplus_{\be \in D + D'} c_\be \odot X^\be}{c_\be \in \bigboxplus_{\be = \bd + \bd'} a_\bd \odot b_{\bd'}} \, .
    \end{align*} 
    In the univariate case, it has the structure of a \emph{superring}~\cite{Ameri+Eyvazi+Hoskova-Mayerova:2019}.
    However even in the univariate case, this structure has a lot of discrepancies such as not being distributive~\cite{Ameri+Eyvazi+Hoskova-Mayerova:2019}, multiplication not being associative, and $\HH[X_1, \dots, X_n]$ not being free in the sense of universal algebra~\cite{Baker+Lorscheid:18}.
\end{remark}

\begin{remark}[Roots of polynomials not closed under taking ideals] \label{rem:hyperideal}
If we allow for multi-valued multiplication, we can still attempt to define an ideal as we usually would for polynomial rings.
We define a \emph{hyperfield polynomial ideal} $I \subseteq \HH[X_1, \dots, X_n]$ to be a subset satisfying
\begin{enumerate}[label=(I\arabic*)]
    \item $\0 \in I$,
    \item If $f \in I$, then $\lambda \boxdot f \subseteq I$ for all $\lambda \in \HH[X_1, \dots, X_n]$,\label{eq:ideal+multiplication}
    \item If $f, g \in I$, then $f \boxplus g \subseteq I$. \label{eq:ideal+addition}
\end{enumerate}
It becomes quickly apparent that such a definition is too coarse for defining varieties over hyperfields.

Over a field, the common roots of a set of polynomials are also roots of all polynomials in the ideal they generate.
This is not true for hyperfield polynomial ideals.
Considering the additive closure in \ref{eq:ideal+addition}, we may have two polynomials $f,g \in I$ with a common root $\ba \in V(f) \cap V(g)$, but there exists some $h \in f \boxplus g \subseteq I$ such that $\ba \notin V(h)$.
This already occurs in simple examples: consider the univariate affine polynomials over $\SS$:
\[
f = X \boxplus \1 \, , \quad g = -X \boxplus -\1 \, , \quad h = X \boxplus -\1 \in f \boxplus g \, .
\]
Both $f$ and $g$ have a unique root $X = -\1$, but the unique root of $h$ is $X = \1$, despite $h \in f \boxplus g$.
A similar phenomenon holds for the multiplicative action in \ref{eq:ideal+multiplication}, where $f \in I$ has a root $\ba \in V(f)$, but there exists some $\lambda \in \HH[X_1, \dots, X_n]$ and $h \in \lambda \boxdot f$ such that $\ba \notin V(h)$.
For example, consider the univariate polynomials over $\SS$:
\[
f = X^2 \boxplus -X \boxplus \1 \, , \quad h = X^3 \boxplus X^2 \boxplus X \boxplus \1 \in (X \boxplus \1) \boxdot f \, .
\]
While $f$ has the unique root $X = \1$, the unique root of $h$ is $X = -\1$.
Similar examples can easily be constructed for other hyperfields including $\TT$.
\end{remark}

\begin{remark}[Hyperfield polynomial ideals are too restrictive] \label{rem:hyperiedal+restrictive}
In the case where $\HH = \TT$, we have a good notion of what tropical varieties and tropical ideals should be~\cite{MR:18}.
We show that the only possible hyperfield polynomial ideals of $\TT[X_1, \dots, X_n]$ are those generated by monomials, and hence are far too restrictive.

Let $f = \bigboxplus_D c_{\bd} \odot \bX^\bd$ be a generator for the polynomial ideal $I \subseteq \TT[X_1, \dots, X_n]$.
Then \ref{eq:ideal+addition} implies that
\[
f \boxplus f = \SetOf{\bigboxplus_{\bd \in D} a_{\bd} \odot \bX^\bd}{a_\bd \in [\infty, c_\bd]} \subseteq I \, .
\]
In particular, the monomial $\bX^\bd$ is contained in $I$ for all $\bd \in D$.
As we can write $f$ as a linear combination of these monomials, we can replace $f$ in the generating set by the monomials $\{\bX^\bd \mid \bd \in D\}$.
Repeating this over all generators, we obtain that $I$ is generated by monomials.
In particular, its associated variety $V(I) \subseteq \TT^n$ can only be a collection of coordinate subspaces.

The situation is even worse when considering hyperfield Laurent polynomial ideals $I \subseteq \TT[X_1^\pm, \dots, X_n^\pm]$, a common setup in tropical geometry.
In this case, each of the generating monomials of $I$ can be inverted and so $I$ is either the zero ideal or all of $\TT[X_1^\pm, \dots, X_n^\pm]$.
\end{remark}

These examples demonstrate that hyperfield polynomial ideals are too large to be compatible with algebraic varieties.
In contrast, collections of polynomials $f_*(J)$ coming from an ideal $J$ over a field are smaller as they are not closed under addition and the multiplicative action of $\HH[X_1, \dots, X_n]$, but their varieties are better behaved as we shall see in Section~\ref{sec:funda_thm}.
It remains unclear what `ideal-like' structure should be used for defining varieties over general hyperfields: we will briefly discuss this in Section~\ref{Sec:further-qs}.

\section{Relatively algebraically closed homomorphisms}\label{sec:rac}

In this section, we recall the definition of a relatively algebraically closed (RAC) homomorphism between hyperfields, first proposed in~\cite{Maxwell:21}.
We then exhibit a number of new families of RAC homomorphisms, including ways of constructing new maps from previously known RAC maps.

To first motivate the definition, we recall the notion of an algebraically closed hyperfield.
\begin{definition}
A hyperfield $\HH$ is called \emph{algebraically closed} if every univariate polynomial has a root in $\HH$. 
\end{definition}

\begin{lemma} \label{lem:surj-alg-clsd}
Let $f:\HH_1 \rightarrow \HH_2$ be a surjective hyperfield homomorphism.
If $\HH_1$ is algebraically closed, then $\HH_2$ is also algebraically closed.
\end{lemma}
\begin{proof}
Consider an arbitrary $p \in \HH_2[X]$.
As $f$ is surjective, there exists $q \in \HH_1[X]$ such that $f_*(q) = p$.
As $\HH_1[X]$ is algebraically closed, there exists some $a \in V(q)$.
By Lemma~\ref{pushforwardgeneral}, $f(a) \in V(f_*(q)) = V(p)$.
\end{proof}

\begin{example} \label{Ex:trop-ph-alg-clsd}
    Given a factor hyperfield $\HH = \FF/U$, the quotient map $\tau\colon\FF \rightarrow \HH$ is surjective.
    Hence, if $\FF$ an algebraically closed field, $\HH$ is also algebraically closed.
    This immediately tells us that many of our examples of hyperfields are algebraically closed, including $\KK, \TT$ and $\PP$.

    We can extend this further to other hyperfields via Lemma \ref{lem:surj-alg-clsd}.
    For example, the identity map $\mathrm{id}: \PP \rightarrow \Phi$ is a surjective hyperfield homomorphism and hence $\Phi$ is also algebraically closed.
\end{example}

Let us briefly consider Lemma~\ref{lem:surj-alg-clsd} for fields.
Let $f \colon \FF_1 \rightarrow \FF_2$ be a surjective homomorphism where $\FF_1$ is algebraically closed and $p \in \FF_1[X]$ an arbitrary polynomial of degree $d$.
As $p$ splits into linear factors over $\FF_1$, homomorphism properties show $f_*(p)$ also splits into linear factors over $\FF_2$,
\[
p = \prod_{i=1}^d (X - a_i) \, \Rightarrow \, f_*(p) = \prod_{i=1}^d (X - f(a_i)) \, .
\]
This gives a much stronger property than simply that $\FF_2$ is algebraically closed, $f$ gives a bijection between the roots of $p$ and $f_*(p)$ counted with multiplicity.
In particular, we have $f(V(p)) = V(f_*(p))$, a strengthening of Lemma~\ref{pushforwardgeneral}.

One key property that makes this work is that polynomials over algebraically closed fields have exactly degree many roots.
As the following example from~\cite{Baker+Lorscheid:18} demonstrates, polynomials over hyperfields may have many more roots than their degree.

\begin{example} \label{ex:infinite+roots}
Consider the degree two polynomial $P(X) = X^2 \boxplus X \boxplus \1 \in \PP[X]$ over the phase hyperfield.
Despite being a degree two polynomial, we show that it has infinitely many roots.
Consider the phase homomorphism $\ph: \CC \rightarrow \PP$, and the family of polynomials $p_k \in \CC[X]$ with root sets
        \[
        p_k(X) = X^2 + X + k \in \CC[X] \; , \; k \in \left[\frac{1}{4}, +\infty \right) \quad , \quad V(p_k) = \left\{\frac{-1 \pm i\sqrt{4k -1}}{2}\right\}  \, .
        \]
    Observe that the pushforward of each polynomial in the phase map is $\ph_*(p_k) = P$, and so does not depend on the choice of $k$.
    However, the image of the roots do depend on $k$, as 
    \[
    \ph(V(p_k)) = \SetOf{e^{i\theta} \in \PP}{\theta \in \tan^{-1}\left(\pm\sqrt{4k -1}\right) \cap \left(\frac{\pi}{2}, \frac{3\pi}{2}\right)} \, .
    \]
    As each of these are roots of $P$, as we vary $k$ we see that $\0 \in P(e^{i\theta})$ for all $\pi / 2 < \theta < 3\pi /2$.
    As such, $P$ has infinitely many roots, all coming from different polynomials in the pre-image of the phase map.
    In particular, there are many roots of $P$ that we cannot lift to a root of $p_k$, despite both $\CC$ and $\PP$ being algebraically closed.
\end{example}

Motivated by this example, we would like to restrict to maps between hyperfields that give a bijection between roots of univariate polynomials.
This motivates the definition of a relatively algebraically closed homomorphism.
\begin{definition}\label{def:RAC}
    Let $f \colon \HH_1 \rightarrow \HH_2$ be a surjective hyperfield homomorphism.
    We say that $f$ is \emph{relatively algebraically closed (RAC)} if for all univariate polynomials $p \in \HH_1[X]$ and every root $\beta \in V(f_\star(p))$, there exists $\alpha \in f^{-1}(\beta)$ such that $\alpha \in V(p)$.
\end{definition}
Note that the RAC property gives the opposite inclusion of Lemma~\ref{pushforwardgeneral} for univariate polynomials, i.e. $V(f_*(p)) \subseteq f(V(p))$.
We will show in Section~\ref{sec:Kapranov} that this extends to multivariate polynomials, giving a hyperfield analogue of Kapranov's theorem.


\begin{example}\label{ex:alg-clsd-kras-RAC}
    Recall that every hyperfield $\HH$ has a trivial homomorphism to $\KK$, and that every univariate polynomial $p \in \KK[X]$ has a root.
    This gives an alternative characterisation of algebraically closed: a hyperfield $\HH$ is algebraically closed if and only if the trivial homomorphism $\omega \colon \HH \rightarrow \KK$ is relatively algebraically closed.
\end{example}
\begin{example}
    Given an algebraically closed valued field $(\FF, \val)$, the valuation map $\val\colon \FF \rightarrow \KK \rtimes \Gamma$ is RAC.
    Proofs of this are given in~\cite[Lemma 3.14]{MB} and~\cite[Proposition 3.1.5]{MS:21}, and in both cases this property is used to prove Kapranov's theorem.
\end{example}
\begin{example}
    The homomorphism $-\log(|\cdot|) \colon \TT\CC \rightarrow \TT$ introduced in Example~\ref{ex:trop+complex+map} is RAC.
    This was proved in~\cite{Maxwell:21} as a new non-trivial example of a RAC map, however we shall reprove this as a corollary in Section~\ref{sec:trop+ext+rac}.
\end{example}


\begin{example}\label{Ex:not-RAC}
    Many natural homomorphisms between hyperfields are \emph{not} RAC, as the following examples demonstrate.
    \begin{itemize}
        \item The sign homomorphism $\sgn: \RR \rightarrow \SS $ is \emph{not} RAC.
        Consider the irreducible polynomial $p(X) = X^2 -X +1 \in \RR[X]$. 
        Its pushforward $\sgn_*(p) = X^2 \boxplus -X \boxplus \1 \in \SS[X]$ has a root at $\1$ despite $p$ having no roots.

        
        
        \item The phase homomorphism $\ph: \CC \rightarrow \PP$ is \emph{not} RAC as Example~\ref{ex:infinite+roots} demonstrates.
    \end{itemize}
\end{example}

\subsection{Relatively algebraically closed maps from fields} \label{sec:field+to+stringent}

In the following, we observe that if $f \colon \FF \rightarrow \HH$ is a surjective homomorphism from an algebraically closed field to a stringent hyperfield, it is necessarily relatively algebraically closed.
This result is deduced from a number of existing results in the literature.

We briefly recall the notion of multiplicity for roots of univariate polynomials over hyperfields.
Recall from Remark~\ref{rem:hyperpoly} that univariate polynomials have a hypermultiplication $\boxdot$ defined on them, giving $\HH[X]$ the structure of a superring. 
Let $p \in \HH[X]$, the \emph{multiplicity} $\mult_a(p)$ of an element $a \in \HH$ is defined as,
\begin{equation}\label{multidef}
\mult_a(p) = 
\begin{cases}
0 &  a \notin V(p) \\ 
1 + \mathrm{max}\SetOf{\mult_a(q)}{p(X) \in (X \boxplus -a) \boxdot q(X)} & a \in V(p)
\end{cases}
\end{equation}
Due to multiplication of polynomials over $\HH$ being multivalued, the polynomial $q(X)$ in \eqref{multidef} is not necessarily unique.
As such, the definition of multiplicity is necessarily recursive; see \cite{Baker+Lorscheid:18} for details on the original definition and examples of the non-uniqueness.

\begin{definition}\label{def:mult+bound}
    A hyperfield $\HH$ satisfies the \emph{multiplicity bound} if for all univariate polynomials $p \in \HH[X]$,
    \begin{equation*}
        \sum_{a \in \HH} \mult_a(p) \leq \deg(p) \, .
    \end{equation*}
    If this inequality is an equality, we say $\HH$ satisfies the \emph{multiplicity equality}.
\end{definition}





With this definition in hand, we make use of the following theorem that gives sufficient conditions for a map to be RAC.

\begin{theorem}\cite[Theorem 4.10]{Maxwell:21} \label{thm:field+hom+rac}
    Let $f\colon \FF \rightarrow \HH$ be a surjective homomorphism.
    If $\FF$ is algebraically closed and $\HH$ satisfies the multiplicity bound, then $f$ is relatively algebraically closed. 
\end{theorem}

This theorem is a special case of a more general statement given in~\cite{Maxwell:21}, where one replaces algebraically closed field with a hyperfield that satisfies the multiplicity equality and the \emph{inheritance property}.
Informally, the inheritance property guarantees one can write a polynomial as a product of linear factors given by its roots with some other lower degree polynomial.
Currently, this property is not well understood over hyperfields, but is trivially satisfied by fields.



In comparison, multiplicities and the multiplicity bound have been much more studied for natural classes of hyperfields~\cite{Liu:20,GrossGunn:23,AGT:23}. 
In particular, we note the following result of Gunn:

\begin{theorem}\cite[Corollary E]{Gunn:22}\label{thm:stringent-multi-bound}
All stringent hyperfields satisfy the multiplicity bound.
\end{theorem}

Combining Theorem~\ref{thm:field+hom+rac} and Theorem~\ref{thm:stringent-multi-bound} yields the following.

\begin{corollary}\label{cor:stringent+rac}
Let $f \colon \FF \rightarrow \HH$ be a surjective homomorphism from an algebraically closed field to a stringent hyperfield.
Then $f$ is relatively algebraically closed.
\end{corollary}

\begin{example}
    This gives an alternative proof that the valuation map $\val \colon \FF \rightarrow \KK \rtimes \Gamma$ on an algebraically closed valued field is RAC.
    However, we can extend this result to fine valuations introduced in Example~\ref{ex:fine+val} to give a new family of RAC maps.
    The field of Hahn series $\hseries{\FF}{\Gamma}$ is algebraically closed when $\FF$ is algebraically closed and $\Gamma$ is a divisible group, hence $\fval \colon \hseries{\FF}{\Gamma} \rightarrow \FF\rtimes \Gamma$ is RAC.
    The same holds for the fine valuation on the field of Puiseux series $\pseries{\FF}$ when $\FF$ is algebraically closed.
\end{example}


\subsection{Tropical extensions of relatively algebraically closed maps}\label{sec:trop+ext+rac}
Given a homomorphism $f\colon \HH_1 \rightarrow \HH_2$, recall its tropical extension $f^\Gamma \colon \HH_1\rtimes \Gamma\rightarrow \HH_2\rtimes \Gamma$ introduced in Definition~\ref{def:hom+extension}.
In this section, we prove that RAC maps are closed under taking tropical extensions.
\begin{proposition}\label{prop:extension+rac}
Let $f\colon \HH_1 \rightarrow \HH_2$ be a relatively algebraically closed homomorphism.
Then for any ordered abelian group $\Gamma$, the induced homomorphism $f^\Gamma\colon \HH_1\rtimes \Gamma \rightarrow \HH_2\rtimes \Gamma$ is relatively algebraically closed.
\end{proposition}
To do so, we first note the following lemma that determines when an element is a root of a polynomial over a tropical extension. 
\begin{lemma}\label{lem:extension+root}
    Let $p \in (\HH\rtimes \Gamma)[X]$ and $a \in \HH\rtimes \Gamma$ with $p(a) = \bigboxplus_{i=0}^n (c_i,g_i)$, where $(c_i,g_i)$ is the evaluation of the monomial of degree $i$ at $a$.
    Then $a$ is a root of $p$ if and only if there exists $J \subseteq \{0,1, \dots, n\}$ such that
\begin{align*}
    g_j < g_i \, , \, g_j = g_{j'} \quad \forall j, j' \in J, i \notin J \, , \qquad \text{ and } \qquad
    \0_{\HH} \in \bigboxplus_{j \in J}{}_\HH \, c_j \, .
\end{align*}
\end{lemma}
\begin{proof}
    Let $J$ be the index set of monomials with $g_j$ minimal: any other monomials are dominated and so do not contribute to the summation.
    The sum of these monomials contains $\0$ if and only if their sum in $\HH$ contains $\0_\HH$.
\end{proof}

\begin{proof}[Proof of Proposition~\ref{prop:extension+rac}]
    We will liberally use the conditions of Lemma~\ref{lem:extension+root} throughout.
    Consider the polynomials
    \[
    p = \bigboxplus_{i=0}^n (\lambda_i,g_i) \odot X^{i} \in (\HH_1\rtimes \Gamma)[X] \quad , \quad f^\Gamma_*(p) = \bigboxplus_{i=0}^n (f(\lambda_i),g_i) \odot X^{i} \in (\HH_2\rtimes \Gamma)[X] \, .
    \]
    Let $(b,h) \in V(f^\Gamma_*(p))$ be a root of $f^\Gamma_*(p)$, and let $J \subseteq \{0, 1, \dots, n\}$ be the indices of the monomials at which the minimum is attained in $\Gamma$, i.e. 
    \begin{align*}
    g_j \odot h^{j} = g_j + j \cdot h < g_i + i \cdot h = g_i \odot h^{i} \quad , \quad g_j \odot h^{j} = g_{j'} \odot h^{j'} \quad , \quad \forall j, j' \in J \, , \, i \notin J \, .
    \end{align*}  
    Then $b$ must be a non-zero root of the polynomial $\bigboxplus_{j \in J} f(\lambda_j) \odot X^{j} \in \HH_2[X]$.
    As $f$ is RAC, there exists a non-zero root $a \in f^{-1}(b)$ of the polynomial $\bigboxplus_{j \in J} \lambda_j \odot X^{j} \in \HH_1[X]$.
    As $(a,h) \in \HH_1\rtimes \Gamma$ satisfies the conditions of Lemma~\ref{lem:extension+root}, it is a root of $p$ in the preimage of $(b,h)$.
\end{proof}

Recall from Example~\ref{ex:alg-clsd-kras-RAC} that a hyperfield is algebraically closed if the trivial homomorphism $\omega\colon \HH \rightarrow \KK$ is relatively algebraically closed.
Applying Proposition~\ref{prop:extension+rac} gives us the following corollary.
\begin{corollary}\label{cor:alg+closed+rac}
Let $\HH$ be an algebraically closed hyperfield.
The homomorphism $\omega^\Gamma \colon \HH\rtimes \Gamma \rightarrow \KK\rtimes \Gamma$ is relatively algebraically closed.
\end{corollary}

\begin{example}
    Recall the map $\eta(z) = -\log(|z|)$ from $\TT\CC$ to $\TT$ introduced in Example~\ref{ex:trop+complex+map}.
    It was shown in~\cite{Maxwell:21} that this map is RAC, and was the only `non-trivial' example known.
    Proposition~\ref{prop:extension+rac} gives us a new perspective on this map via tropical extensions.
    Recall that we can view the map $\eta$ as the tropical extension
    \begin{align*}
    \omega^\Gamma \colon \Phi\rtimes \RR &\rightarrow \KK\rtimes \RR \\
    (\theta, g) &\mapsto (\1,g) \, .
    \end{align*}
    Example~\ref{Ex:trop-ph-alg-clsd} shows that the tropical phase hyperfield is algebraically closed, hence Corollary~\ref{cor:alg+closed+rac} shows $\eta$ must also be RAC.
\end{example}

\subsection{Quotients of relatively algebraically closed maps}
In the following, we show that RAC maps are closed under `compatible' quotients of the domain and target hyperfields.

\begin{proposition}\label{prop:RAC-factor-RAC}
    Let $f:\HH_1 \rightarrow \HH_2$ be a relatively algebraically closed homomorphism.
    If $U_1 \subseteq \HH_1^{\times}$ and $U_2 \subseteq \HH_2^{\times}$ such that $f(U_1) = U_2$, then the map 
    \begin{align*}
    \hat{f}: \HH_1/U_1 & \rightarrow \HH_2/U_2 \\
    \overbar{a} & \mapsto \overbar{f(a)} \
    \end{align*}
    is relatively algebraically closed.
\end{proposition}
\begin{proof}
Let $\tau:\HH_1 \rightarrow \HH_1/U_1$ and $\sigma: \HH_2 \rightarrow \HH_2/U_2$ denote the corresponding quotient maps.
Observe that $\hat{f}$ is defined such that the following diagram commutes,
\[
\begin{tikzcd}
\HH_1 \arrow{r}{f} \arrow{d}{\tau} & \HH_2 \arrow{d}{\sigma} \\
\HH_1/U_1 \arrow{r}{\hat{f}} & \HH_2/U_2
\end{tikzcd}
\]
i.e. we have $\hat{f}(\tau(a)) = \sigma(f(a))$ for all $a \in \HH_1$. 
    We show that $\hat{f}$ is well-defined similarly to the proof of Lemma~\ref{lem:factor+from+hom}.
    Let $a,b \in \HH_1$ such that $\tau(a) = \tau(b)$, i.e. there exists $u_1 \in U_1$ such that $a = b \odot_1 u_1$. Then, 
    \begin{align*}
    \hat{f}(\tau(a)) &= \sigma(f(a)) = \sigma(f(b \odot_1 u_1)) = \sigma(f(b) \odot_2 f(u_1)) \\
    &= \sigma(f(b)) \odot_2 \sigma(f(u_1)) = \sigma(f(b)) \odot_2 \1_2 = \sigma(f(b)) = \hat{f}(\tau(b)),
    \end{align*}
    as $f(u_1) \in U_2$.
    Thus, $\hat{f}$ is independent of the choice of representative for the coset.
   Moreover, it is straightforward to verify that $\hat{f}$ is also a surjective homomorphism.

    To show it is RAC, let $Q \in (\HH_1/U_1)[X]$ and $P = \hat{f}_*(Q) \in (\HH_2/U_2)[X]$.
    We show that for some arbitrary root $z \in V(P)$, there exists $x \in V(Q)$ such that $\hat{f}(x) = z$.
    By \cite[Corollary 6.3.4]{Maxwell:Thesis}, we can decompose the root set $V(P)$ as follows
    \[
    V(P) = \bigcup_{\sigma_*(p) = P} \sigma(V(p)) \, , \quad p \in \HH_2[X] \, . 
    \]
    This implies that there exists a polynomial $p \in \sigma_*^{-1}(P)$ with root $w \in V(p) \subseteq \HH_2$ such that $\sigma(w) = z$.
    We now claim there exists a polynomial $q \in \tau_*^{-1}(Q) \subseteq \HH_1[X]$ such that $f_*(q) = p$.
    To justify this, note that
    \[
    \sigma_*(p) = P = \hat{f}_*(Q) = \hat{f}_*(\tau_*(q')) = \sigma_*(f_*(q')) \, ,
    \]
    for all $q' \in \tau_*^{-1}(Q)$.
    Fixing some $q'$, this implies that each coefficient of $p$ differs from the corresponding coefficient of $f_*(q')$ by a scalar belonging to $U_2$.
    As $f(U_1) = U_2$, there exists $q \in \HH_1[X]$ such that $f_*(q) = p$ and has coefficients that differ from $q'$ by scalars in $U_1$.
    Hence $\tau_*(q') = \tau_*(q)$, and so $q$ satisfies the claim.
    
    As $f$ is RAC, for $w \in V(p)$ there exists $y \in f^{-1}(w)$ such that $y \in V(q)$.
    Applying Lemma~\ref{pushforwardgeneral}, we have
    \[
    \tau(y) \in \tau(V(q))  \subseteq V(\tau_*(q)) = V(Q) \, .
    \]
    Setting $x = \tau(y)$, we see that
    \[
    \hat{f}(x) = \hat{f}(\tau(y)) = \sigma(f(y)) = \sigma(w) = z \, ,
    \]
    which concludes the proof.
\end{proof}

\begin{example}
    As the complex numbers are algebraically closed, the trivial homomorphism $\omega\colon\CC \rightarrow \KK$ is RAC.
    By applying Proposition~\ref{prop:RAC-factor-RAC} when $U_1 = \RR_{>0}$ and $U_2 = \1_\KK$, we get that $\PP \cong \CC/\RR_{>0} \rightarrow \KK$ is also a RAC map and an alternative proof that $\PP$ is algebraically closed.
    
The compatible quotient condition, i.e. $f(U_1) = U_2$, cannot be relaxed and still preserve the RAC property, as the following example demonstrates.
The identity map $\rm{id} \colon \CC \rightarrow \CC$ is trivially RAC.
    However, if we try to quotient by $U_1 = \1_\CC$ and $U_2 = \RR_{>0}$, we get the phase map $\ph: \CC \rightarrow \PP$, which is not RAC as shown in Example~\ref{Ex:not-RAC}.
\end{example}

\begin{example} \label{ex:hyperfield+valuations}
    A number of authors~\cite{Kra:56,KLS,Linzi:23} have investigated valuations on hyperfields.
    Similar to the discussion in Section~\ref{subsec:homo+val}, these can also be viewed as homomorphisms $\HH \rightarrow \KK\rtimes \Gamma$ where $\Gamma$ is the value group.
    If $\HH = \FF/U$ is a factor hyperfield where $(\FF,\val)$ a valued field and $U \subseteq \val^{-1}(0)$ then it is straightforward to verify induced map $\val_U$ on $\HH$ is also a valuation.
    Moreover, as $\val$ is RAC, applying Proposition~\ref{prop:RAC-factor-RAC} in the case $U_1 = U$ and $U_2 = \1_2$ implies that $\val_U$ must also be RAC.
\end{example}


\section{Tropical geometry for hyperfields}\label{sec:main+thms}
In this section, we prove our main theorems: hyperfield generalisations of Kapranov's theorem and the Fundamental theorem of tropical geometry.

\subsection{Kapranov's theorem} \label{sec:Kapranov}
In this section we prove a generalisation of Kapranov's theorem for RAC homomorphisms between hyperfields.
We first prove it for affine hypersurfaces, and show the projective and torus cases as corollaries.

\begin{theorem}\label{thm:Kapranov}
Let $f\colon \HH_1 \rightarrow \HH_2$ be a relatively algebraically closed homomorphism and $p \in \HH_1[X_1, \dots, X_n]$ a polynomial.
Then
\[
V(f_*(p)) = f(V(p)) \, .
\]
\end{theorem}
The proof of this statement is very similar to the case where $\HH_2 = \TT$ given in~\cite{Maxwell:21}, as we do not require $\HH_2$ to have any additional properties.
However, we include a proof for completeness.
\begin{proof}
Let $p = \bigboxplus_{I \in \ZZ_{\geq 0}^n} c_I \odot \bX^I \in \HH[X_1, \dots, X_n]$ and pick some root $\bb \in V(f_*(p))$ of the push-forward.
Given Lemma~\ref{pushforwardgeneral}, it suffices to show there exists $\ba \in V(p)$ such that $f(\ba) = \bb$.

Fix some $\lambda_i \in f^{-1}(b_i)$: note that such values exist as $f$ is surjective.
For any $D = (d_1, \dots, d_n) \in \ZZ_{\geq 0}^n$, we define the maps
\begin{align*}
    \phi_D\colon \HH_1 &\rightarrow \HH_1^n & \psi_D\colon \HH_2 &\rightarrow \HH_2^n \\
    x &\mapsto (\lambda_1 \odot_1 x^{d_1}, \dots, \lambda_n \odot_1 x^{d_n}) & y &\mapsto (f(\lambda_1) \odot_2 y^{d_1}, \dots, f(\lambda_n) \odot_2 y^{d_n}) \, .
\end{align*}
It is easy to verify by hyperfield homomorphism properties that the following diagram commutes.
\[
\begin{tikzcd}
\HH_1 \arrow{r}{f} \arrow{d}{\phi_D} & \HH_2 \arrow{d}{\psi_D} \\
\HH_1^n \arrow{r}{f} & \HH_2^n
\end{tikzcd}
\]
Rather than pulling $\bb$ back through the bottom row to find a root, we will instead traverse the other way around the square.

First note that $\psi_D(\1_2) = \bb$.
Consider $\phi_D^*(p) = p \circ \phi_D$, the pullback of the polynomial $p$ through $\phi_D$, i.e.
\[
\phi_D^*(p) = \bigboxplus_{I \in \ZZ_{\geq 0}^n} c_I \odot_1 (\lambda_1 \odot_1 X^{d_1})^{i_1} \odot_1 \cdots \odot_1 (\lambda_n \odot_1 X^{d_n})^{i_n} = \bigboxplus_{I \in \ZZ_{\geq 0}^n} c_I \odot_1 \blambda^I \odot_1 X^{D\cdot I} \, .
\]
Note that $\phi_D^*(p)$ may not be a polynomial in $\HH_1[X]$, as we require coefficients to be elements rather than sets.
There may exist two support vectors $I,I'$ of $p$ such that $D\cdot I = D \cdot I'$, hence the hypersum of their coefficients may give a set.
However, as $p$ has finite support we can choose $D$ sufficiently generically such that $D\cdot I \neq D \cdot I'$ for all $I, I'$ in the support of $p$.
This ensures $\phi_D^*(p) \in \HH_1[X]$, as well as ensuring we cannot get any cancellation of terms.

By expanding out and using properties of homomorphisms, we see that 
\[
f_*(p)(\psi_D(X)) = f_*(\phi_D^*(p))(X) \quad \Rightarrow \quad \0_2 \in f_*(p)(\bb) = f_*(\phi_D^*(p))(\1_2) \, .
\]
As $f$ is a RAC homomorphism, there exists an element $\tilde{a} \in \HH_1$ such that $\0_1 \in \phi_D^*(p)(\tilde{a})$ and $f(\tilde{a}) = \1_2$.
Define $\ba = \phi_D(\tilde{a})$, we then see that:
\begin{align*}
    p(\ba) &= \phi_D^*(p)(\tilde{a}) \ni \0_1 \, , \\
    f(\ba) &= (f(\lambda_1)\odot_2 f(\tilde{a})^{d_1}, \dots, f(\lambda_n)\odot_2 f(\tilde{a})^{d_n}) = \bb \, .
\end{align*}
\end{proof}

We can deduce a projective version of Kapranov's theorem that follows immediately from Theorem~\ref{thm:Kapranov} and \eqref{eq:proj+hyp}.
\begin{corollary} \label{cor:proj+kapranov}
    Let $f\colon \HH_1 \rightarrow \HH_2$ be a relatively algebraically closed homomorphism and $p \in \HH_1[X_0,X_1, \dots, X_n]$ a homogeneous polynomial.
    Then
    \[
    PV(f_*(p)) = f(PV(p)) \, .
    \]
\end{corollary}
%

We can also deduce a torus version of Kapranov's theorem as a fairly immediate corollary.
\begin{corollary} \label{cor:torus+kapranov}
    Let $f\colon \HH_1 \rightarrow \HH_2$ be a relatively algebraically closed homomorphism and $p \in \HH_1[X_1^\pm, \dots, X_n^\pm]$ a Laurent polynomial.
    Then
    \[
    V^\times(f_*(p)) = f(V^\times(p)) \, .
    \]
\end{corollary}
\begin{proof}
We will utilise the projective case from Corollary~\ref{cor:proj+kapranov}, restricted to the torus charts $U^\times_{\HH_1}$ and $U^\times_{\HH_2}$ using Lemma~\ref{lem:torus+as+proj}.
%
From this, we obtain
\[
f(V^\times(p)) = f(PV(\overbar{p_{\aff}}) \cap U^\times_{\HH_1}) = f(PV(\overbar{p_{\aff}})) \cap U^\times_{\HH_2} = PV(f_*(\overbar{p_{\aff}})) \cap U^\times_{\HH_2} = V^\times(f_*(p)) \, .
\]
\end{proof}

\subsection{Extending the fundamental theorem of tropical geometry to hyperfields}\label{sec:funda_thm}

In tropical geometry, Kapranov's theorem can be extended to the fundamental theorem by considering varieties defined by polynomial ideals, rather than hypersurfaces defined by a single polynomial.
The aim of this section is to present a parallel picture for RAC hyperfield homomorphisms.
However, there are major deficiencies with polynomial ideals in the hyperfield setting as laid out in Remarks~\ref{rem:hyperpoly}, \ref{rem:hyperideal} and \ref{rem:hyperiedal+restrictive}.
As such, we restrict our setting to RAC maps from fields where polynomial ideals are well behaved with respect to algebraic varieties.

\begin{theorem}\label{thm:fund+thm+proj}
    Let $f: \FF \rightarrow \HH$ be a relatively algebraically closed homomorphism from an algebraically closed field $\FF$ to a hyperfield $\HH$.
    Then, for any homogeneous ideal $I \subseteq \FF[X_0,X_1 \dots X_n]$, 
    \[
    f(PV(I)) = PV(f_*(I)) \, .
    \]
\end{theorem}

As with Kapranov's theorem, one containment follows directly from hyperfield homomorphism properties, and doesn't require any (relatively) algebraically closed assumptions.
Moreover, we will show the affine case, as the projective and torus versions of the statement follow immediately.
\begin{lemma}\label{lem:ideal-var-incl}
    Let $f: \FF \rightarrow \HH$ be a homomorphism from a field $\FF$ to a hyperfield $\HH$.
    For any ideal $I \subseteq \FF[X_1, \dots X_n]$, 
    \[  
    f(V(I)) \subseteq V(f_*(I)) \, .
    \]
\end{lemma}
\begin{proof}
    Consider $\bb \in f(V(I))$, and pick some $\ba \in f^{-1}(\bb)$ such that $\ba \in V(I)$.
    In particular, we have $\ba \in V(p)$ for all $p \in I$.
    Lemma \ref{pushforwardgeneral} implies that $f(\ba) \in f(V(p)) \subseteq V(f_*(p))$ for all $p \in I$, and hence $\bb \in \bigcap_{p \in I} V(f_*(p)) = V(f_*(I))$.
\end{proof}

The other containment is more involved.
We will use multiple facts from elementary algebraic geometry without proof: we refer to~\cite{CLO} for further details.
We will also make liberal use of Lemma~\ref{pushforwardgeneral} and Theorem~\ref{thm:Kapranov}.

\begin{proof}
    The inclusion $f(PV(I)) \subseteq PV(f_*(I))$ follows from the affine case in Lemma~\ref{lem:ideal-var-incl} and quotienting by scalars.    
    We prove the reverse inclusion via induction on the dimension of the variety $PV(I)$. 
    
    Let $\dim(PV(I)) = 0$, then $PV(I) = \{\ba^{(1)}, \dots \ba^{(l)} \} \subset \PP^{n}(\FF)$ is a finite set of $l$ points.
    Consider any $\by \in \PP^{n}(\HH) \setminus f(PV(I))$, we construct a polynomial in $f_*(I)$ that vanishes on $\{f(\ba^{(1)}), \dots , f(\ba^{(l)}) \}$ but not on $\by$.
    For each $\ba \in PV(I)$, fix $j \in \{0, \dots, n\}$ such that $f(a_j) \neq \0$.
    Then there exists $k \in \{0, \dots ,n \}\setminus j$ such that 
    \[
    \frac{f(a_k)}{f(a_j)} \neq \frac{y_k}{y_j} \quad \Rightarrow \quad \frac{a_k}{a_j} \neq \frac{w_k}{w_j}  \quad \forall \bw \in f^{-1}(\by) \subseteq \PP^n(\FF) \, .
    \]
    Note that these quantities are well defined in projective space.
    Define the linear polynomial $P_\ba = a_j \cdot X_k - a_k \cdot X_j$ and let $P = \prod_{\ba \in PV(I)} P_\ba \in \FF[X_0, \dots, X_n]$.
    This is defined such that $P(\ba) = 0$ for all $\ba \in PV(I)$, but $P(\bw) \neq 0$ for all $\bw \in f^{-1}(\by)$.
    By the Nullstellensatz, there exists $m \in \NN$ such that $P^m \in I$, but still does not vanish on $f^{-1}(\by)$.
    Applying Theorem~\ref{thm:Kapranov}, we see that
    \begin{align*}
        \by = f(\bw) & \notin f(PV(P^m)) = PV(f_*(P^m))  \quad \Rightarrow \quad \by \notin \bigcap_{p \in I}PV(f_*(p)) = PV(f_*(I)) \, .
    \end{align*}
    This completes the base case of $\dim(PV(I)) = 0$.

    We now assume that the claim holds for all varieties of dimension less than $k$, and let $\dim(PV(I)) = k$.
    We will first prove the case where $PV(I)$ is irreducible i.e. $I$ is prime, and consider the reducible case after.
    
    Consider an arbitrary $\by \in \PP^n(\HH) \setminus f(PV(I))$ and fix some element $\bw \in f^{-1}(\by)$ in the preimage with corresponding maximal ideal $\mathfrak{m}_\bw \subseteq \FF[X_0, X_1,\dots, X_n]$.
    As $\mathfrak{m}_\bw$ is maximal and $\dim(PV(I)) > 0$, there exists some $q \in \mathfrak{m}_\bw \setminus I$.
    Geometrically, this is equivalent to $q(\bw) = 0$ but $q$ does not vanish on all of $PV(I)$.
    By \cite[Proposition 9.4.10]{CLO}, we have
    \[
    \dim(PV(I) \cap PV(q)) = \dim(PV(I + \langle q\rangle) = k-1 \, .
    \]
    Applying the inductive hypothesis gives us
    \[
    \by = f(\bw) \notin f(PV(I)) \supseteq f(PV(I + \langle q\rangle)) = \bigcap_{p \in I +\langle q \rangle } PV(f_*(p)) \, .
    \]
    Hence, there exists $p \in I +\langle q \rangle$ such that $\0 \notin f_*(p)(\by)$, and therefore that $p(\bw) \neq 0 $ for all $\bw\in f^{-1}(\by)$.
    By expressing $p$ as a certain combination, we note
    \begin{align*}
    p &= \hat{p} + \lambda q^m \quad , \quad \hat{p} \in I \, , \, \lambda \in \FF[X_0,\dots ,X_n] \, , \, m \in \NN \\
    \Rightarrow \, p(\bw) &= \hat{p}(\bw) + \lambda(\bw) q^m(\bw) = \hat{p}(\bw) \neq 0 \, .
    \end{align*}
    Hence, there exists a polynomial $\hat{p} \in I$ such that $\bw \notin PV(\hat{p})$ for all $\bw \in f^{-1}(\by)$.
    Applying Theorem~\ref{thm:Kapranov} gives us 
    \[
    \by = f(\bw) \notin f(PV(\hat{p})) = PV(f_*(\hat{p})) \quad \Rightarrow \quad \by \notin PV(f_*(\hat{p})) \supseteq \bigcap_{p \in I} PV(f_*(p)) = PV(f_*(I)) \, .
    \]

    Finally, suppose that $PV(I) = \bigcup_{s=1}^r PV(I_s)$ is reducible into irreducible components $PV(I_s)$.
    On each irreducible component, we have $f(PV(I_s)) = PV(f_*(I_s))$.
    Therefore, applying Lemma~\ref{lem:pushforward+reducible} gives us
    \[
    f(PV(I)) = \bigcup_{s=1}^r f(PV(I_s)) = \bigcup_{s=1}^r PV(f_*(I_s)) = PV(f_*(I)) \, .
    \]
    
\end{proof}

\begin{lemma} \label{lem:pushforward+reducible}
    Let $f: \FF \rightarrow \HH$ be a relatively closed hyperfield homomorphism from an algebraically closed field $\FF$ to a hyperfield $\HH$. 
    Let $PV(I)$ be a projective variety with decomposition into irreducible components $PV(I) = PV(I_1) \cup \dots \cup PV(I_r)$.
    Then
    \[
    PV(f_*(I)) = PV(f_*(I_1)) \cup \dots \cup PV(f_*(I_r)) \, .
    \]
\end{lemma}
\begin{proof}
    Assume $\by \in PV(f_*(I_s))$ for some $s \in [r]$, then $f_*(p)(\by) \ni \0$ for all $p \in I_s$.
    As $I \subseteq I_s$, it immediately follows $\by \in PV(f_*(I))$.

    Conversely, if $\by \notin \bigcup_{s=1}^r PV(f_*(I_s))$, then for each $s \in [r]$ there exists some $p_s \in I_s$ such that $f_*(p_s)(\by) \notni \0$.
    This implies that $p_s(\bw) \neq 0$ for all $\bw \in f^{-1}(\by)$.
    Let $p = \prod_{i=1}^r p_s \in \bigcap_{s=1}^r I_s$: this implies $p$ is contained in the radical of $I$, and therefore, there exists $m \in \NN$ such that $p^m \in I$.
    Moreover, $p^m(\bw) \neq 0$ for all preimages $\bw$, and so applying Kapranov's Theorem gives
    \[
    \by = f(\bw) \notin f(PV(p^m)) = PV(f_*(p^m)) \supseteq PV(f_*(I)) \, .
    \]
\end{proof}

We get analogous statements for affine varieties and torus subvarieties by applying Lemmas~\ref{lem:aff+as+proj} and~\ref{lem:torus+as+proj} respectively.
\begin{corollary}
    Let $f: \FF \rightarrow \HH$ be a relatively algebraically closed homomorphism from an algebraically closed field $\FF$ to a hyperfield $\HH$.
    Then, for any ideal $I \subseteq \FF[X_1 \dots X_n]$, 
    \[
    f(V(I)) = V(f_*(I)) \, .
    \]
    Moreover, for any Laurent ideal $J \subseteq \FF[X_1^\pm, \dots, X_n^\pm]$,
    \[
    f(V^\times(J)) = V^\times(f_*(J)) \, .
    \]
\end{corollary}
\begin{proof}
We denote the affine chart over $\FF$ and $\HH$ as $U_\FF$ and $U_\HH$ respectively.
By analogous arguments to those from Corollary~\ref{cor:torus+kapranov}, observe that
\[
f(PV(\overbar{I}) \cap U_\FF) = f(PV(\overbar{I})) \cap U_\HH \quad , \quad f(PV(\overbar{J_{\aff}}) \cap U^\times_\FF) = f(PV(\overbar{J_{\aff}})) \cap U^\times_\HH \, .
\]
Combining Lemma~\ref{lem:aff+as+proj} and Theorem~\ref{thm:fund+thm+proj} with this observation, we see
    \[
    f(V(I)) = f(PV(\overbar{I}) \cap U_\FF) = f(PV(\overbar{I})) \cap U_\HH = PV(f_*(\overbar{I})) \cap U_\HH = V(f_*(I)) \, .
    \]
    Similarly, combining Lemma~\ref{lem:torus+as+proj} and Theorem~\ref{thm:fund+thm+proj} with the above observation gives the analogous statement for torus subvarieties,
    \[
    f(V^\times(J)) = f(PV(\overbar{J_{\aff}}) \cap U^\times_\FF) = f(PV(\overbar{J_{\aff}})) \cap U^\times_\HH = PV(f_*(\overbar{J_{\aff}})) \cap U^\times_\HH = V^\times(f_*(J)) \, .
    \]
\end{proof}

\section{Fine tropical varieties}\label{sec:fine+tropical+varieties}

Let $(\FF,\val)$ be an algebraically closed valued field and $I \subseteq \FF[X_1, \dots, X_n]$ an ideal.
The fundamental theorem of tropical geometry states we can define the corresponding \emph{tropical variety} $\trop(V(I))$ as either $\val(V(I))$, the image of the algebraic variety in the valuation map, or as $V(\val_*(I))$, the hyperfield variety determined by the induced ideal over $\TT$.
For certain valued fields, we can consider a fine valuation map that recalls more information and is relatively algebraically closed by Corollary~\ref{cor:stringent+rac}.
Moreover, Theorem~\ref{thm:fund+thm+proj} gives us a natural notion of a \emph{fine tropical variety} derived from this fine valuation.
In this section, we introduce fine tropical varieties and motivate them as an avenue of further study.

Recall from Examples~\ref{ex:hahn+series} and~\ref{ex:fine+val} the field of Hahn series $\hseries{\FF}{\Gamma}$.
We will restrict ourselves to $\hseries{\CC}{\RR}$ from now on as this is the most natural setting for tropical geometry, but the following will hold for any algebraically closed field $\FF$ and divisible ordered abelian group $\Gamma$.
Recall that $\hseries{\CC}{\RR}$ comes with the fine valuation map
    \begin{align*}
        \fval\colon \hseries{\CC}{\RR} &\rightarrow \CC\rtimes \RR \\ \rho := \sum_{g \in G} c_g t^g &\mapsto (c_\gamma , \gamma) \, , \, \gamma = \min(g \mid g \in G) \, ,
    \end{align*}
that forgets everything except the leading term of a Hahn series.
We denote the leading coefficient of $\rho$ by $\lc(\rho) := c_\gamma$.
We will utilise that the fine valuation of a non-zero Hahn series $\rho$ can be written as $\fval(\rho) = (\lc(\rho), \val(\rho))$.

\begin{definition}\label{def:fine+trop+variety}
Let $I \subseteq \hseries{\CC}{\RR}[X_1, \dots, X_n]$ be an ideal.
The associated (affine) \emph{fine tropical variety} is
\[
\ftrop(V(I)) = \fval(V(I)) = V(\fval_*(I)) \subseteq (\CC \rtimes \RR)^n \, .
\]
The projective and torus analogues are defined similarly.
\end{definition}

The Fundamental theorem of tropical geometry~\cite[Theorem 3.2.5]{MS:21} gives a third description of a tropical variety via Gr\"obner theory.
While this is not possible for enriched valuations in general, we can give a description of a fine tropical variety this way.
We shall restrict ourselves to subvarieties of the torus for the remainder of this section to simplify exposition and highlight the parallels with ordinary tropical geometry.

We very briefly recall some Gr\"obner theory over the valued field $\hseries{\CC}{\RR}$, see~\cite{MS:21} for derivations of definitions over general valued fields.
The residue field of $\hseries{\CC}{\RR}$ is $\CC$, where the representative of the Hahn series $\rho$ is its leading coefficient $\lc(\rho)$.
As such, given some Laurent polynomial $p = \sum \rho_\bd \cdot \bX^\bd \in \hseries{\CC}{\RR}[X_1^\pm , \dots, X_n^\pm]$, we define its \emph{initial form} with respect to $\bu \in \RR^n$ to be the Laurent polynomial
\[
\initial_\bu(p) = \sum_{\bd \in D_{\min}} \lc(\rho_\bd) \cdot \bX^\bd \in \CC[X_1^\pm, \dots, X_n^\pm] \, , \quad D_{\min} := \SetOf{\bd \in \ZZ^n}{\val(\rho_\bd) + \bd\cdot \bu \text{ minimal }} \, .
\]
Given an ideal $I \subseteq\hseries{\CC}{\RR}[X_1^\pm, \dots, X_n^\pm]$, its \emph{initial ideal} $\initial_\bu(I)$ with respect to $\bu \in \RR^n$ is 
\[
\initial_\bu(I) = \langle \initial_\bu(p) \mid p \in I\rangle \subseteq \CC[X_1^\pm, \dots, X_n^\pm] \, ,
\]
the ideal generated by all initial forms of polynomials in $I$.

\begin{theorem}\label{thm:fine+grobner}
    Let $I \subseteq \hseries{\CC}{\RR}[X_1^\pm, \cdots, X_n^\pm]$ be an ideal.
    The fine tropical variety associated to $I$ is
    \[
    \ftrop(V^\times(I)) = \bigcup_{\bu \in \RR^n} \left(V^\times(\initial_\bu(I)) \times \{\bu\}\right) \subseteq (\CC^\times \times \RR)^n \, .
    \]
\end{theorem}
\begin{proof}
    Throughout we will only make use of the definition $\ftrop(V^\times(I)) := V^\times(\fval_*(I))$ of a fine tropical variety, and so write this everywhere.
    
    We first show that for a single polynomial $p = \sum_{\bd \in D} \rho_\bd \cdot \bX^\bd$, we have
    \begin{equation} \label{eq:fine+hypersurface}
    V^\times(\fval_*(p)) = \bigcup_{\bu \in \RR^n} \left(V^\times(\initial_\bu(p)) \times \{\bu\}\right)
    \end{equation}
    Consider some $(\bz,\bu) \in (\CC^\times)^n \times \RR^n \subseteq (\CC \rtimes \RR)^n$.
    Utilising the hyperfield operations of $\CC \rtimes \RR$, we have
    \begin{align*}
    (\bz,\bu) \in V^\times(\fval_*(p)) \, \Leftrightarrow \, \infty \in \bigboxplus_{\bd \in D} \fval(\rho_\bd) \odot (\bz,\bu)^{\odot \bd} \, \Leftrightarrow \, \infty \in \bigboxplus_{\bd \in D_{\min}} \fval(\rho_\bd) \odot (\bz^\bd,\bd\cdot\bu)
    \end{align*}
    where we can restrict the summation to $D_{\min}$ as other terms will not contribute.
    Restricting to the $\CC$-part of this sum, we obtain that
    \begin{align*}
    \infty \in \bigboxplus_{\bd \in D_{\min}} \fval(\rho_\bd) \odot (\bz^\bd,\bd\cdot\bu) \, \Leftrightarrow \, \sum_{\bd \in D_{\min}} \lc(\rho_\bd) \cdot \bz^\bd = 0  \, \Leftrightarrow \, \bz \in V^\times(\initial_\bu(p)) \, ,
    \end{align*}
    giving the equivalence in \eqref{eq:fine+hypersurface}.
    
    We now consider the general case for an ideal $I$.
    Using \eqref{eq:fine+hypersurface}, we have
    \begin{align*}
        V^\times(\fval_*(I)) = \bigcap_{p \in I} V^\times(\fval_*(p)) = \bigcup_{\bu \in \RR^n} \left(\bigcap_{p \in I} V^\times(\initial_\bu(p)) \times \{\bu\}\right) \, .
    \end{align*}
    Hence suffices to show that $V^\times(\initial_\bu(I)) = \bigcap_{p \in I} V^\times(\initial_\bu(p))$.
    This follows from the fact that $\{\initial_\bu(p) \mid p \in I\}$ generates $\initial_\bu(I)$.
\end{proof}

\begin{remark}
    Consider the homomorphism $\omega^\RR \colon \CC \rtimes \RR \rightarrow \KK \rtimes \RR \cong \TT$, the tropical extension of the trivial homomorphism $\omega$ by $\RR$.
    We can view this as a forgetful morphism that forgets the $\CC$-data.
    Moreover, it relates $\fval$ and $\val$ in the following commutative diagram:
    \[
    \begin{tikzcd}[row sep=huge, column sep=large]
    \hseries{\CC}{\RR} \arrow[r, "\fval"] \arrow[rd, "\val"] & \CC \rtimes \RR \ar[d, "\omega^\RR"]\\
    & \KK \rtimes \RR
    \end{tikzcd}
    \]
    Utilising Definition~\ref{def:fine+trop+variety}, the image of a fine tropical variety in $\omega^\RR$ is the underlying tropical variety, i.e.
    \[
    \omega^\RR(\ftrop(V^\times(I)) = \omega^\RR(\fval(V^\times(I)) = \val(V^\times(I)) = \trop(V^\times(I)) \, .
    \]
    As such, Theorem~\ref{thm:fine+grobner} recovers the description of a tropical variety in terms of initial ideals from the Fundamental theorem \cite[Theorem 3.2.5]{MS:21}, namely 
    \[
    \trop(V^\times(I)) = \SetOf{\bu \in \RR^n}{\initial_\bu(I) \neq \langle 1 \rangle} \, .
    \]
    This follows as the variety $V^\times(\initial_\bu(I))$ is non-empty if and only if $\initial_\bu(I) \neq \langle 1 \rangle$.
\end{remark}

\begin{remark} \label{rem:initial+ideal+stable}
    It seems at first that Theorem~\ref{thm:fine+grobner} implies that to compute a fine tropical variety we would have to compute infinitely many initial ideals.
    However, there are only finitely many initial ideals of $I$ that can occur and they are related by a polyhedral complex structure called the \emph{Gr\"obner complex} of $I$.
    Furthermore, the tropical variety $\trop(V^\times(I))$ is a subcomplex of the Gr\"obner complex of $I$, and hence inherits a polyhedral complex structure from it; see~\cite{MS:21} for full details.
    In particular, for any polyhedral cell $\sigma \in \trop(V^\times(I))$ and $\bu, \bv$ in the relative interior of $\sigma$, we have $\initial_\bu(I) = \initial_\bv(I)$.
    As such, the complex variety $V^\times(\initial_\bu(I))$ is invariant as we range over the relative interior of the polyhedral cell $\sigma$.
    Coupled with Theorem~\ref{thm:fine+grobner}, we can view a fine tropical variety as a tropical variety with a complex variety attached to the relative interior of each polyhedral cell.
\end{remark}

\begin{example}\label{ex:fine+variety}
    Consider the polynomial $P = X + Y - 1 \in \hseries{\CC}{\RR}[X^\pm,Y^\pm]$.
    This defines the line in the plane 
    \[
    V^\times(P) = \SetOf{(\rho, 1- \rho)}{\rho \in \hseries{\CC}{\RR}^\times \setminus \{1\}} \subseteq (\hseries{\CC}{\RR}^\times)^2  \, .
    \]
    We will construct its fine tropical variety $\ftrop(V^\times(P))$ in two different ways, following Theorem~\ref{thm:fund+thm+proj}.
    A schematic of $\ftrop(V^\times(P))$ is given in Figure~\ref{fig:fine+line}.
    
    Let $\rho = c_\gamma t^\gamma + c_\beta t^\beta + \cdots$ where $\gamma < \beta < g$ for all $g \in G \setminus \{\gamma,\beta\}$.
    The fine valuation of the point $(\rho, 1- \rho) \in V(P)$ depends entirely on these leading terms, and can be considered a first order approximation of the point:
    \begin{align}\label{eq:fine+line+1}
    \left(\rho, 1- \rho\right) &= \begin{cases}
        (c_\gamma t^\gamma + O(t^\beta) \; , \, 1 + O(t^\gamma)) & \gamma > 0 \\
        (c_\gamma + O(t^\beta) \; , \, 1 - c_\gamma + O(t^\beta)) & \gamma = 0 , c_\gamma \neq 1 \\
        (1 + O(t^\beta) \; , \, -c_\beta t^\beta + O(t^g)) & \gamma = 0 , c_\gamma = 1 \\
        (c_\gamma t^\gamma + O(t^\beta) \; , \, -c_\gamma t^\gamma + O(t^{\min(\beta,0)})) & \gamma < 0 \\
        \end{cases} \nonumber \\
    \Rightarrow \quad \fval\left(\rho, 1- \rho\right) &= 
    \left\lbrace\begin{array}{@{}l@{\quad}l@{\qquad}l@{}}
        \fp{c_\gamma, \gamma}{1,0} & \gamma > 0 & (A) \\
        \fp{c_\gamma, 0}{1 - c_\gamma ,0} & \gamma = 0 , c_\gamma \neq 1 & (B) \\
        \fp{1, 0}{c_\beta ,\beta} & \gamma = 0 , c_\gamma = 1 & (C) \\
        \fp{c_\gamma, \gamma}{-c_\gamma, \gamma} & \gamma < 0 & (D) \\
    \end{array} \right.
    \end{align}
    Ranging over all points in $V^\times(P)$ gives the fine tropical variety $\ftrop(V^\times(P)) = \fval(V^\times(P))$.
    Equation~\eqref{eq:fine+line+1} is labelled to identify cases with Figure~\ref{fig:fine+line} and the components of $V^\times(\fval_*(P))$ given in~\eqref{eq:fine+line+2}.

    Alternatively, we can consider the `fine tropical polynomial' $\fval_*(P) = X \boxplus Y \boxplus (-1,0) \in (\CC \rtimes \RR)[X^\pm, Y^\pm]$.
    The solutions to this polynomial are those where the `complex' component gives a solution when restricted to monomials where the `tropical' component attains the minimum, i.e.
    \begin{equation}\label{eq:fine+line+2}
    \begin{split}
    \0 &\in \fval_*(P)\fp{c_X,g_X}{c_Y,g_Y} \\
    &= (c_X,g_X) \boxplus (c_Y,g_Y) \boxplus (-1,0)
    \end{split}
    \, \Longleftrightarrow \, \begin{cases}
        g_X > g_Y = 0 \, , \, c_Y - 1 = 0 & (A) \\
        g_X = g_Y = 0 \, , \, c_X + c_Y - 1 = 0 & (B) \\
        g_Y > g_X = 0 \, , \, c_X - 1 = 0 & (C) \\
        g_X = g_Y < 0 \, , \, c_X + c_Y = 0 & (D)
    \end{cases}
    \end{equation}
    The set of all points $\fp{c_X,g_X}{c_Y,g_Y} \in (\CC \rtimes \RR)^2$ that satisfy one of these conditions also gives the fine tropical variety $\ftrop(V^\times(P)) = V^\times(\fval_*(P))$.
    Equation~\eqref{eq:fine+line+2} is labelled to identify cases with Figure~\ref{fig:fine+line} and the components of $\fval(V(P))$ given in~\eqref{eq:fine+line+1}.    
    
    Finally, we note that as described in Theorem~\ref{thm:fine+grobner}, the complex hypersurface attached to the point $\bu \in \RR^2$ is the variety $V^\times(\initial_{\bu}(P))$ cut out by the initial form $\initial_{\bu}(P)$.
    Moreover, these initial forms and complex hypersurfaces are constant on the relative interior of the polyhedral cells of $\trop(V^\times(P))$ as described in Remark~\ref{rem:initial+ideal+stable}.
\begin{figure}[ht]
        \centering
        \includegraphics{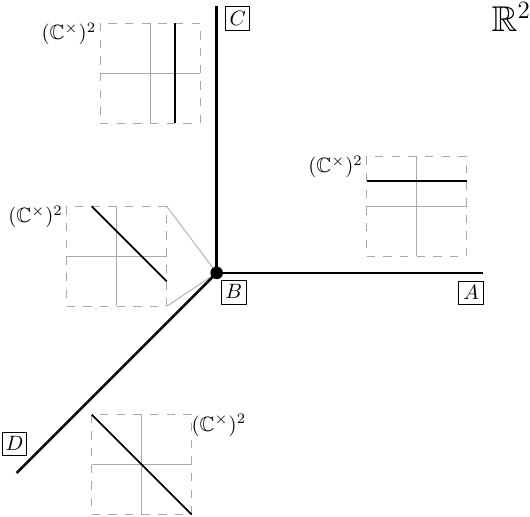}
        \caption{The fine tropical line $\ftrop(V^\times(P))$ from Example~\ref{ex:fine+variety}.
We view the ambient space $(\CC \rtimes \RR)^2$ as $\RR^2$ with a copy of $(\CC^\times)^2$ at each point.
As such, $\ftrop(V^\times(P))$ can be viewed as the usual tropical line in $\RR^2$ with the subvariety of the complex torus $V^\times(\initial_\bu(P))$ attached at each point $\bu \in \RR^2$.
The labels allow identification with the components given in~\eqref{eq:fine+line+1} and~\eqref{eq:fine+line+2}.
}
        \label{fig:fine+line}
    \end{figure}    
\end{example}


\begin{remark}
Recall that the field of Puiseux series $\pseries{\CC}$ also carries a fine valuation whose image is $\CC \rtimes \QQ$.
Hence we can define fine tropical varieties for ideals over the Puiseux series in the space $(\CC \rtimes \QQ)^n$.
For ordinary tropical varieties, one could take the closure in the Euclidean topology to view the variety in $\RR^n$ rather than $\QQ^n$.
However, there are a number of topological concerns when doing this, especially outside of the setting of ordinary tropical geometry: see~\cite{Joswig+Smith:23} for a discussion of pitfalls in the higher rank setting as an example.
As such, we will only view fine tropical varieties over the Puiseux series in the space $(\CC \rtimes \QQ)^n$.
\end{remark}

\begin{remark}
In the last decade, there have been advances in viewing tropical geometry through a scheme-theoretic lens.
The Giansiracusa approach to tropical scheme theory~\cite{GG:2016} defines tropical schemes via congruences of the semiring of tropical polynomials, namely the \emph{bend relations}.
Such congruences give rise to tropical ideals, a special family of ideals of the semiring of tropical polynomials, as detailed in \cite{MR:18}.
Alternative approaches to tropical scheme theory include Lorscheid's theory of ordered blueprints, of which the Giansiracusa tropicalization can be seen as a special case of; see~\cite{Lorscheid:22} for details.

Currently, our definition of a fine tropical variety has no scheme-theoretic analogue.
A natural first approach is to try and generalise the Giansiracusa bend relations to $\CC \rtimes \RR$ and its set of polynomials.
However, the immediate issue is that the bend relations are a semring congruence, and unlike the tropical hyperfield, it is not clear that one can associate a semiring to the hyperfield $\CC \rtimes \RR$.
For example, the tropical semiring is obtained from the tropical hyperfield by defining the single-valued addition 
\[
a \oplus b = 
\begin{cases}
a \boxplus b & a \neq b \\ a & a = b \, .
\end{cases}
\]
Note that the only multi-valued sum we have to alter is the sum of inverses $a \boxplus a = [a, \infty]$, where $a$ is the canonical choice of element as the minimal element.
This construction does not work for $\CC \rtimes \RR$: the sum of additive inverses
\[
(c,g) \boxplus (-c,g) = \{(b,h) \mid b \in \CC^\times \, , \, h \in \RR \, , \, h > g\} \cup \{\0\} \
\]
contains no `minimal' or canonical element to set $(c,g) \oplus (-c,g)$ to.

One may be able to recast the Giansiracusa approach in the language of hyperfields to avoid this transition. 
A positive result towards this is that there is a one-to-one correspondence between congruences and ideals over a hyperring~\cite[Theorem A]{Jun:21}.
However, as described in Remark~\ref{rem:hyperpoly}, the set of polynomials over a hyperfield is not a hyperring, and so its congruences may have many of the same issues as its ideals do; see Remarks~\ref{rem:hyperideal} and \ref{rem:hyperiedal+restrictive}.
Alternatively, one may be able realise fine tropical schemes in the framework of ordered blueprints and Lorscheid tropical scheme theory, though this remains unclear.
\end{remark}

The remainder of this section will discuss some possible applications of fine tropical varieties as motivation for their further study.

\paragraph{Stable intersections}\label{sec:stable}


A fundamental aspect of tropical geometry is that the intersection of two tropical varieties may not be a tropical variety.
Stable intersection turns out to be the correct notion of intersection, where one perturbs the tropical varieties before intersecting them.
Formally, it is defined as
\[
\trop(V(I)) \wedge \trop(V(J)) = \lim_{\epsilon \rightarrow 0} \left(\trop(V(I)) \cap (\trop(V(J)) + \epsilon\bv)\right)
\]
for some generic vector $\bv \in \RR^n$.

In \cite{Joswig+Smith:23}, a novel approach to stable intersection was introduced, where perturbation was performed in a higher rank tropical semiring, intersected set-theoretically and then projected to the usual rank one tropical semiring.
In the following, we provide evidence that a similar paradigm holds by passing to a generic fine tropical variety.

\begin{example}\label{ex:intersection}
Consider the bivariate polynomials
\begin{align*}
P &= X + Y - 1 \, , &
Q &= tX + (1+t^2) Y + 1 \, , & P,Q \in \hseries{\CC}{\RR}[X,Y] \, .
\end{align*}
Their corresponding hypersurfaces $V(P), V(Q)$ are lines in the plane, and so their intersection is a single point, namely 
\[
V(P) \cap V(Q) = \left(\frac{2+t^2}{1-t+t^2}, \frac{-1-t}{1-t+t^2}\right) = \left( 2 + 2t + t^2 + O(t^3), -1 -2t -t^2 + O(t^3) \right) \in \hseries{\CC}{\RR}^2 \, .
\]
Naively, we would expect the intersection of their tropical hypersurfaces to be the single point $(0,0)$.
However, this is not the case as Figure~\ref{fig:stable+intersection} demonstrates: their intersection is a one-dimensional ray, but their stable intersection is the correct point,
\[
\trop(V(P)) \cap \trop(V(Q)) = \{(g, 0) \mid g \geq 0\} \, , \quad  \trop(V(P)) \wedge \trop(V(Q)) = \{(0,0)\} \, .
\]

\begin{figure}
\begin{center}
\includegraphics[width=0.45\textwidth]{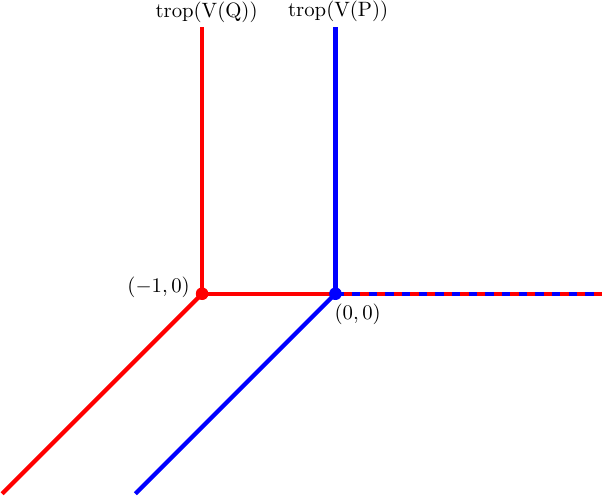} \quad
\includegraphics[width=0.45\textwidth]{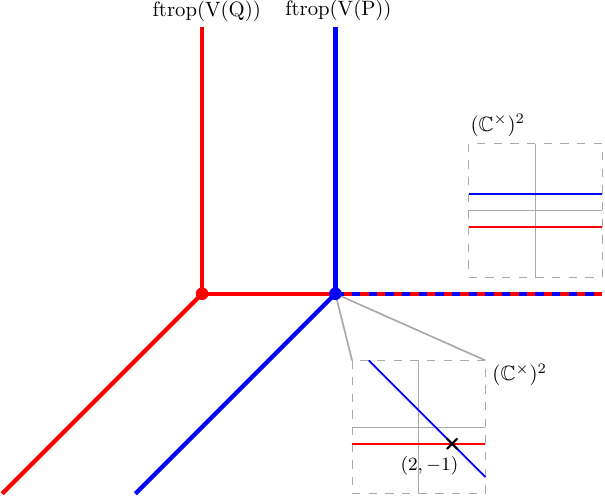}
\end{center}
\caption{Left: The tropical hypersurfaces $\trop(V(P))$ and $\trop(V(Q))$ from Example~\ref{ex:intersection}.
Note that they intersect in a 1-dim ray, but they intersect stably at a single point.\\
Right: The fine tropical hypersurfaces $\ftrop(V(P))$ and $\ftrop(V(Q))$ from the same example.
These do intersect transversally, and so they intersect at a single point.}
\label{fig:stable+intersection}
\end{figure}

Instead let us consider their fine tropical hypersurfaces: $\ftrop(V(P))$ was computed in Example~\ref{ex:fine+variety}, and $\ftrop(V(Q))$ is computed analogously:
\begin{align*}
\ftrop(V(P)) &= \SetOf{\fp{c_X,g_X}{c_Y,g_Y} \in (\CC \rtimes \RR)^2}{
\begin{array}{l}g_X > g_Y = 0 \, , \, c_Y - 1 = 0 \\
g_Y > g_X = 0 \, , \, c_X - 1 = 0 \\
g_X = g_Y < 0 \, , \, c_X + c_Y = 0 \\
g_X = g_Y = 0 \, , \, c_X + c_Y - 1 = 0
\end{array}}\\
\ftrop(V(Q)) &= \SetOf{\fp{c_X,g_X}{c_Y,g_Y} \in (\CC \rtimes \RR)^2}{
\begin{array}{l}
g_X > - 1 \, , \, g_Y = 0 \, , \, c_Y + 1 = 0 \\
g_X = -1 \, , \, g_Y > 0 \, , \, c_X + 1 = 0 \\
g_X + 1 = g_Y < 0 \, , \, c_X + c_Y = 0 \\
g_X +1 = g_Y = 0 \, , \, c_X + c_Y + 1 = 0
\end{array}} 
\end{align*}
It is not hard to check that these two hypersurfaces intersect at a single point $\fp{2,0}{-1,0}$.
Moreover, this point is precisely $\fval(V(P) \cap V(Q))$, and hence its image under the projection $\omega^\RR \colon \CC \rtimes \RR \rightarrow \KK \rtimes \RR \cong \TT$ is the stable intersection of $\trop(V(P))$ and $\trop(V(Q))$, i.e.
\[
\trop(V(P)) \wedge \trop(V(Q)) = \omega^\RR\left(\ftrop(V(P)) \cap \ftrop(V(Q))\right) \, .
\]
Observe from Figure~\ref{fig:stable+intersection} that the two rays that intersected in infinitely many points in $\TT$ no longer intersect at all.
\end{example}

We note that moving to the fine tropical variety does not sidestep the need to perturb from all cases, as we may still encounter non-generic intersections.
However, we do gain two advantages from this perspective,
Firstly, non-generic intersections happen less frequently for fine tropical varieties as $\CC \rtimes \RR$ is in some sense a `bigger' space than $\TT$.
Secondly, if we do need to perturb the fine tropical varieties, we expect it suffices to only perturb in $\CC$ by sufficiently generic complex numbers, and leave the `tropical' part as is.
On the level of Hahn series, this corresponds to perturbing the leading coefficient while leaving the leading exponent be.

We can make this perturbation precise as follows.
Let $P = \sum_{\bd \in D} \rho_\bd \cdot \bX^\bd \in \hseries{\CC}{\RR}[X_1^\pm, \dots, X_n^\pm]$ be a Laurent polynomial with support $D \subseteq \ZZ^n$, i.e. $\rho_\bd \neq 0$ if and only if $\bd \in D$.
Given some $\alpha \in (\CC^\times)^D$, we define the perturbed polynomial $P^\alpha = \sum_{\bd \in D} \alpha_\bd \cdot \rho_\bd \cdot \bX^\bd$.
Note that $\rho_\bd$ and $\alpha_\bd \cdot \rho_\bd$ have the same leading exponents despite being different Hahn series.
As such, the tropical varieties $\trop(V(P))$ and $\trop(V(P^\alpha))$ are equal, while their fine tropical varieties are distinct.

\begin{conjecture} \label{conj:stable+intersection}
    Let $P, Q \in \hseries{\CC}{\RR}[X_1^\pm, \dots, X_n^\pm]$ where $P$ has support $D \subseteq \ZZ^n$.
    Then there exists a Zariski open subset $A \subseteq (\CC^\times)^D$ such that
    \[
    \trop(V(P)) \wedge \trop(V(Q)) = \omega^\RR(\ftrop(V(P^\alpha)) \cap \ftrop(V(Q))) 
    \]
    for all $\alpha \in A$.
\end{conjecture}

\begin{example}
Consider the polynomial $P' = X + Y + 1$, and recall the polynomial $Q = tX + (1+t^2) Y + 1$ from Example~\ref{ex:intersection}.
Despite $P$ and $P'$ having the same tropical variety, they give rise to different fine tropical varieties.
In particular, the fine tropical variety $\ftrop(V(P'))$ associated to $P'$ does not have generic intersection with $\ftrop(V(Q))$:
\[
\ftrop(V(P')) \cap \ftrop(V(Q)) = \SetOf{\fp{c_X,g_X}{-1,0}}{g_X > 0 \, , \, c_X \in \CC^\times} \, .
\]
Moreover, the image of their intersection under $\omega^\RR$ is not equal to the stable intersection $\trop(V(P')) \wedge \trop(V(Q))$.

By picking $\alpha = (1,1,-1) \in (\CC^\times)^3$, we can perturb $P'$ to $(P')^\alpha = P$.
It follows from Example~\ref{ex:intersection} that the intersection $\ftrop(V((P')^\alpha)) \cap \ftrop(V(Q))$ is generic, and that its image under $\omega^\RR$ is equal to the stable intersection $\trop(V(P')) \wedge \trop(V(Q))$.
However, Conjecture~\ref{conj:stable+intersection} claims that almost any $\alpha \in (\CC^\times)^3$ should suffice for this perturbation.
Indeed, it can be checked that any $\alpha \in A$ gives the desired (stable) intersection, where
\[
A = \SetOf{(\alpha_X, \alpha_Y, \alpha_1) \in (\CC^\times)^3}{ \alpha_Y \neq \alpha_1} \, .
\]
\end{example}

\paragraph{Polyhedral homotopies}\label{sec:homotopy}


Polyhedral homotopy continuation is a method for finding solutions to systems of polynomial equations introduced by Huber and Sturmfels~\cite{Huber+Sturmfels:95}: we follow the tropical description in~\cite{Leykin+Yu:19}.
Given a polynomial system $P = (p_1, \dots, p_n) \subset \CC[X_1, \dots, X_n]$ with finitely many solutions, we wish to numerically approximate $V(P)$.
This is done by perturbing the coefficients of $P$ by multiplying through by $t^\gamma$ for various generic $\gamma \in \QQ$.
This gives a perturbed system of equations over the Puiseux series $P_t = (p_1(t), \dots, p_n(t)) \subset \pseries{\CC}[X_1, \dots, X_n]$.
If one can understand the solutions of $P_t$ around $0 < t \ll 1$, then we can `track' these solutions to an approximate solution of the original system $P = P_1$ via homotopy methods.
We note that the method given in~\cite{Leykin+Yu:19} is more general than this, allowing us to find the solutions on a fixed algebraic variety, but we restrict to ordinary case for simplicity.

The method for finding an initial solution of $P_t$ for small $t$ is roughly as follows.
By calculating the tropical variety $\trop(V(P_t))$, we obtain the leading exponents of the Puiseux series solutions.
Moreover, for each $\bu \in \trop(V(P_t))$ one can also recover the leading coefficient of the solution by computing the initial ideal $\initial_\bu(\langle P_t\rangle)$.
Generically, the initial ideal should be zero-dimensional, and so the unique point $\bc \in V^\times(\initial_\bu(\langle P_t\rangle))$ gives us a first-order approximation $(c_1t^{u_1}, \dots, c_nt^{u_n})$ of a solution in $V(P_t)$.
Ranging over all $\bu \in \trop(V(P_t))$ gives us a first-order approximation of our solution set $V(P_t)$ for $t$ close to zero, which in general should be sufficient information to begin homotopy continuation.

By the characterisation given in Theorem~\ref{thm:fine+grobner}, we can reframe this initial solution computation as computing the fine tropical variety $\ftrop(V(P_t))$.
This motivates fine tropical varieties as a natural object of study when computing initial solutions for polyhedral homotopy continuation, and there are two avenues we hope this perspective will be useful.
Firstly, computing the initial solution is a fairly expensive computation, requiring at least a mixed volume computation which is $\#P$-hard, if not a Gr\"obner basis computation which is worse-case doubly exponential.
As such, any algorithmic advantages we can gain from the structure of fine tropical varieties, even in specific instances, would be welcome.
Secondly, Theorem~\ref{thm:fine+grobner} allows us to compute and encode initial solutions of polynomial systems with a positive dimensional solution set in only finite information.
Explicitly, even if $\trop(V(P_t))$ has positive dimension, Remark~\ref{rem:initial+ideal+stable} shows that the initial ideal associated to each point are constant on cells.
As such, we can describe the fine tropical variety $\ftrop(V(P_t))$ in finitely many polyhedral cells, each with a single initial ideal associated.
We hope that this can be utilised alongside higher-dimensional continuation methods~\cite{Sommese+Wampler:2005} to numerically approximate positive-dimensional varieties.

\section{Further questions} \label{Sec:further-qs}
We end with some further questions and directions for study.

Theorem~\ref{thm:fund+thm+proj} gives a fundamental theorem for RAC maps from fields to hyperfields.
As laid out in Remarks~\ref{rem:hyperpoly}, \ref{rem:hyperideal} and \ref{rem:hyperiedal+restrictive}, extending this statement to RAC maps between hyperfields is currently not possible due to no coherent notion of a polynomial ideal that is compatible with algebraic varieties for general hyperfields.
As such, the following question would have to be addressed before on could extend the fundamental theorem further:
\begin{question}
Can one formulate a natural notion of a polynomial ideal $I \subseteq \HH[X_1, \dots, X_n]$ that is compatible with algebraic varieties over certain well-behaved $\HH$?
\end{question}
Over $\TT$, Maclagan and Rinc\'on~\cite{MR:18,MR:22} introduced the notion of a \emph{tropical ideal}, an ideal of the semiring of tropical polynomials which has an additional monomial elimination axiom, similar to vector elimination axioms for matroids that occur for polynomial ideals over $\FF$.
This axiom allows one to consider tropical ideals as compatible `layers' of matroids over $\TT$.
One possible approach is to generalise this notion to a wider family of hyperfields, where $\HH$-ideals are compatible layers of matroids over $\HH$, in the sense of \cite{Baker+Bowler:19}.
There are two issues that would have to be overcome for this approach.
Firstly, not all hyperfields satisfy vector elimination axioms~\cite{LAnderson}, so this technique will not hold in full generality, but perhaps work for natural families of hyperfields.
Secondly, semiring ideals differ quite dramatically from hyperfield ideals: \cite[Example 5.14]{MR:18} demonstrates that arbitrary ideals of the semiring of tropical polynomials are not restrictive enough and can give rise to non-polyhedral tropical varieties.
This is the opposite problem to hyperfield polynomial ideals, where Remark~\ref{rem:hyperiedal+restrictive} shows the naive definition is far too restrictive to be useful.

As another direction of generalisation, we use the relatively algebraically closed property to prove Kapranov's theorem (Theorem~\ref{thm:Kapranov}).
However, this is a stronger property than we may actually require.
As an example, consider the absolute value map from the complex numbers to the \emph{triangle hyperfield}:
\begin{align*}
| \cdot | \colon \CC &\rightarrow \Delta = (\RR_{\geq 0}, \boxplus, \odot) & a \boxplus b &= \SetOf{c}{|a-b| \leq c \leq a+b} \\
z &\mapsto |z| & a \odot b &= a\cdot b 
\end{align*}
This is a hyperfield homomorphism that is not RAC.
However, it was shown in~\cite{Purbhoo:08} that Kapranov's theorem does hold in this setting.
The difference is one cannot just consider the polynomial $p$, but rather the whole principal ideal $\langle p \rangle$, as there are many `higher polynomials' that contain information that $p$ does not.
This motivates the following:
\begin{question}
For which maps $f\colon \FF \rightarrow \HH$ do we have $f(V(\langle p \rangle)) = V(f_*\langle p \rangle)$?
\end{question}
As with Theorem~\ref{thm:fund+thm+proj}, this restricts us to maps from fields until we have a good handle on which hyperfields have well-defined notions of (principal) ideals.

Finally, note that in Theorem~\ref{thm:fund+thm+proj} we do not give any conditions on $\HH$ aside from it is the target of a RAC map from an algebraically closed field.
In practice, all examples of such hyperfields that we know about are stringent.
\begin{question}
Classify all hyperfields $\HH$ that are the target of a relatively algebraically closed map from a field.
Is $\HH$ necessarily stringent?
\end{question}
The connection between RAC maps and the multiplicity bound (Definition~\ref{def:mult+bound}) was investigated in~\cite{Maxwell:21}, in particular to derive some sufficient conditions for a map to be RAC.
However, if the RAC map is from a field, it may be \emph{necessary} that the target hyperfield satisfies the multiplicity bound.
Key examples of non-stringent hyperfields, such as $\PP$ and $\Delta$, are known to exceed the multiplicity bound.
As such, one avenue to investigate this question is to verify whether non-stringent hyperfields exceed the multiplicity bound in general.


\bibliographystyle{plain}

\end{document}